\documentclass[10pt, reqno]{amsart}

\usepackage{amssymb}
\usepackage{tikz}
\usetikzlibrary{calc}
\usetikzlibrary{decorations.pathreplacing}
\usetikzlibrary{patterns}
\usepackage{bm}

\usepackage{tikzscale}
\usepackage{filecontents}

\newtheorem{theorem}{Theorem}[section]
\newtheorem{proposition}[theorem]{Proposition}
\newtheorem{lemma}[theorem]{Lemma}

\theoremstyle{definition}
\newtheorem{definition}[theorem]{Definition}

\newtheorem{remark}[theorem]{Remark}
\newtheorem{note}[theorem]{Note}

\theoremstyle{remark}

\numberwithin{equation}{section}

\everymath{\displaystyle}

\begin{document}

\title{On the trace of unimodal L\'{e}vy processes on Lipschitz domains}

\author{Gavin Armstrong}
\address{Department of Mathematics, University of Oregon, 
Eugene, OR 97403}
\email{gka@uoregon.edu}


\maketitle

\begin{abstract}
We show that the second term in the asymptotic expansion as $t \to 0$ of the trace of the Dirichlet heat kernel on Lipschitz domains for unimodal L\'{e}vy processes, satisfying some weak scaling conditions, is given by the surface area of the boundary of the domain. This brings the asymptotics for the trace of unimodal L\'{e}vy processes in domains of Euclidean space on par with those of symmetric stable processes as far as boundary smoothness is concerned.
\end{abstract}

\section{Introduction}

The following two-term estimate for the trace of the heat kernel corresponding to the symmetric $\alpha$-stable processes, $\alpha \in (0,2]$, on an $R$-smooth domain $D \subset \mathbb{R}^{d}$ was given by Ba\~{n}uelos and Kulczycki \cite{Ban1}:
\begin{equation}
 \left| Z_{D}(t) - \frac{C_{1} |D|}{t^{d/\alpha}} + \frac{C_{2} t^{1/\alpha} |\partial D|}{t^{d/\alpha}} \right| \leq \frac{C_{3} t^{2/\alpha} |D|}{R^{2} t^{d/\alpha}}. \label{1}
\end{equation}

\noindent Ba\~{n}uelos et al. \cite{Ban2} expanded this idea, in analogy with a result for Brownian motion in Brown \cite{Brown}, to bounded Lipschitz domains:
\begin{equation}
t^{d/\alpha} Z_{D}(t) = C_{1} |D| - C_{\mathbb{H}} t^{1/\alpha} \mathcal{H}^{d - 1}(\partial D)+ o \left(t^{1/\alpha} \right). \label{2}
\end{equation}

\noindent In another direction, this first bound (\ref{1}) was generalized by Bogdan and Siudeja \cite{Bog5} to unimodal L\'{e}vy processes satisfying certain weak lower and upper scaling conditions on $R$-smooth domains:
\begin{equation}
\left| Z_{D}(t) - p_{t}(0) |D| + C_{\mathbb{H}}(t) \left| \partial D \right| \right| \leq \frac{c_{\theta} p_{t}(0) T(t)^2 | D |}{R^2}. \label{3}
\end{equation}

\noindent In this paper we combine the results of Ba\~{n}uelos et al. \cite{Ban2} and Bogdan and Siudeja \cite{Bog5} to obtain generalizations of both (\ref{2}) and (\ref{3}). This generalization says that for a unimodal L\'{e}vy processes on a bounded Lipschitz domain we have
\begin{equation}
\left| Z_{D}(t) - p_{t}(0) |D| + C_{\mathbb{H}}(t) \mathcal{H}^{d-1}(\partial D) \right| \leq c(\varepsilon) T(t)^{1-d},
\end{equation}
where $c(\varepsilon) \to 0$ and $\varepsilon \to 0$.

\*

\section{Preliminaries}

We call a measure {\bf isotropic} if it is absolutely continuous on $\mathbb{R}^d \backslash \{0\}$ with respect to Lebesgue measure and is invariant under linear isometries of $\mathbb{R}^d$. We call a measure {\bf isotropic unimodal}, or {\bf unimodal} in short, if its density function is also radially non-increasing. A L\'{e}vy process is called {\bf isotropic unimodal} if all its density functions are isotropic unimodal, see \cite{Bog3,Wat}. Unimodal L\'{e}vy processes are characterized by L\'{e}vy-Khintchine (characteristic) exponents of the form
\begin{equation}
\psi(\xi) = \sigma^{2} |\xi|^{2} + \int_{\mathbb{R}^{d}} \left( 1 - \cos \langle \xi, x \rangle \right) \nu(dx),
\end{equation}
where $\nu(dx) = \nu(x) dx = \nu(|x|) dx$ is a unimodal L\'{e}vy measure and $\sigma \geq 0$. 
Since $\psi(\xi)$ is a radial function, we often let $\xi(r) = \psi(\xi)$ where $\xi \in \mathbb{R}^{d}$ and $r = |\xi| \geq 0$.

In what follows, we assume that we have a unimodal L\'{e}vy measure and we consider the pure-jump, $\sigma = 0$, L\'{e}vy process $X = \left(X_t \right)_{t \geq 0}$ on $\mathbb{R}^d$ determined by the L\'{e}vy-Khintchine formula:
\begin{equation}
\mathbb{E}e^{i \langle \xi, X_t \rangle} = \int_{\mathbb{R}^d} e^{i \langle \xi, x \rangle} p_{t}(dx) = e^{-t\psi(\xi)}.
\end{equation}

\noindent Here $p_{t}(dx)$ is the distribution of $X_{t}$. It turns out that $p_{t} (dx)$ is also unimodal; therefore we can call the process $X$ (isotropic) unimodal. We wish for $p_{t}(dx)$ to have bounded and smooth density functions, $p_{t}(x)$ for $t > 0$. This is achieved as a consequence of the Hartman-Wintner condition, see Lemma~1.1 of \cite{Bog2},
\begin{equation}
\lim_{|\xi| \to \infty} \psi(\xi)/\ln(\xi) = \infty. \label{Hart}
\end{equation}

\noindent The Hartman-Wintner condition itself will be a consequence of our assumption that $\psi(\xi)$ satisfies some weak lower scaling condition, yet to be defined. We always assume that the L\'{e}vy-Khintchine exponent, $\psi(\xi)$, is unbounded, that is, $\nu \left( \mathbb{R}^{d} \right) = \infty$. Clearly $\psi(0) = 0$ and $\psi(u) > 0$ for $u > 0$.

\*

\subsection{Renewal function of the ladder-height processes} \*\vspace{0.1in}

Let $X_{t}^{1}$ be the first coordinate process of $X_t$. We define the {\bf running maximum} of $X_{t}$ by 
\begin{equation}
M_t = \sup_{0 \leq s \leq t} X_{s}^{1}.
\end{equation}
We define $L^{0}(t)$ to be the {\bf local time} of $M_t - X_t^{1}$ at $0$. That is, $L^{0} (t)$ is the amount of time, up to time $t$, that $M_t - X_{t}^{1}$ spends at $0$: 
\begin{equation}
L^0(t) = \int_0^t \delta \left( M_s - X_s^{1} \right) ds,
\end{equation}
where $\delta(\cdot)$ is the Dirac delta function.
 Consider the right-continuous inverse of $L^{0}(t)$, $\left( L^{0} \right)^{-1}(s)$, this is called the {\bf the ascending ladder time process} for $X_{t}^{1}$. Composing $X^{1}_{t}$ with $\left( L^{0} \right)^{-1}(s)$ gives us the {\bf ascending ladder-height process}:
\begin{equation}
H_{s} = X^{1}_{\left(L^{0} \right)^{-1} (s)} = M_{\left(L^{0} \right)^{-1}(s)}.
\end{equation}

\noindent The {\bf accumulated potential} of our ascending ladder-height process is then defined by 
\begin{equation}
V(x) = \mathbb{E} \int_{0}^{\infty} 1_{[0,x]} \left( H_{s} \right) ds = \int _{0}^\infty \mathbb{P} \left(H_{s} \leq x \right) ds.
\end{equation}

\noindent The function $V(x)$ is continuous and strictly increasing from $[0,\infty)$ onto $[0,\infty)$. In particular, $\lim_{r \to \infty} V(r) = \infty$ and $V(x)$ is sub-additive: 
\begin{equation}
V(x + y) \leq  V(x) + V(y), \hspace{0.2in} \text{ for all } x, y \in \mathbb{R}.
\end{equation}

\noindent For example, if $\psi(\xi) = |\xi|^{\alpha}$ with $\alpha \in (0,2)$, then $V(x) = x_{+}^{\alpha/2}$. See Example 3.7 in \cite{Song}. For more details on the ascending ladder-height process and accumulated potential see \cite{Bog3} and \cite{Sil}.

\begin{remark}
The relationship between $V(x)$ and $\psi(x)$ is given in Lemma~1.2 of \cite{Bog2} by
$$ V^2(r) \simeq \frac{1}{\psi (1/r)}, \ r > 0. $$
\end{remark}

\noindent The notation ``$\simeq$'' means that there is some constant $C \in (0,\infty)$ such that for all $r > 0$ we have
$$C^{-1} V^{2}(r) \leq \frac{1}{\psi(1/r)} \leq C V^{2}(r).$$
It also worth noting that throughout this paper we use many different constants. The value of these constants is not usually of importance and the same specific constant is rarely required more than once. Hence the letter ``$C$'' is often used generically to refer to a constant, but it almost never refers to the same constant more than once.

\*

\subsection{Scaling} \* \vspace{0.1in}

We are interested in the (relative) power-type behavior of $\psi(r)$ at infinity.

\begin{definition}
We say that $\psi(r)$ satisfies the {\bf weak lower scaling condition at infinity}, $WLSC \left(\underline{\alpha}, \underline{\theta}, \underline{C} \right)$, if there are numbers $\underline{\alpha} > 0$, $\underline{\theta} \geq 0$, and $\underline{C} \in (0,1]$ such that 
\begin{equation}
\psi(\lambda r) \geq \underline{C} \lambda^{\underline{\alpha}} \psi(r), \label{WLSC}
\end{equation}
for $\lambda \geq 1$, $r > \underline{\theta}$. In general, we write $\psi \in WLSC \left(\underline{\alpha}, \underline{\theta}, \underline{C} \right)$.
\end{definition}

\noindent Or, in short, we write $\psi \in WLSC \left(\underline{\alpha}, \underline{\theta}, \underline{C} \right)$, $\psi \in WLSC \left(\underline{\alpha}, \underline{\theta} \right)$, or $\psi \in WLSC \left(\underline{\alpha} \right)$ depending on how specific we want to be. Further, we say that $\psi(r)$ satisfies the {\bf global} weak lower scaling condition at infinity ({\bf global} $WLSC$) if $\psi \in WLSC(\underline{\alpha}, 0)$. If $\underline{\theta} > 0$, then we can emphasize this by calling the scaling ``{\bf local} at infinity''.

\begin{definition}
We say that $\psi(r)$ satisfies the {\bf weak upper scaling condition at infinity}, $WUSC \left(\overline{\alpha}, \overline{\theta}, \overline{C} \right)$, if there are numbers $\overline{\alpha} < 2$, $\overline{\theta} \geq 0$, and $\overline{C} \in [1, \infty)$ such that 
\begin{equation}
\psi(\lambda r) \leq \overline{C} \lambda^{\overline{\alpha}} \psi(r), \label{WUSC}
\end{equation}
for $\lambda \geq 1$, $r > \overline{\theta}$. In general, we write $\psi \in WUSC \left(\overline{\alpha}, \overline{\theta}, \overline{C} \right)$.
\end{definition}

\noindent Or, in short, we write $\psi \in WUSC \left(\overline{\alpha}, \overline{\theta}, \overline{C} \right)$, $\psi \in WUSC \left(\overline{\alpha}, \overline{\theta} \right)$, or $\psi \in WUSC \left(\overline{\alpha} \right)$ depending on how specific we want to be. Further, we say that $\psi(r)$ satisfies the {\bf global} weak upper scaling condition at infinity ({\bf global} $WUSC$) if $\psi \in WUSC(\overline{\alpha}, 0)$. If $\overline{\theta} > 0$, then we can emphasize this by calling the scaling ``{\bf local} at infinity''.

\begin{remark}
As pointed out in Remark 1.4 of \cite{Bog2}, by inflating (or deflating) $\underline{C}$ and $\overline{C}$ we can deflate (or inflate) $\underline{\theta}$ and $\overline{\theta}$ so that $\theta = \underline{\theta} = \overline{\theta} > 0$ in both $WLSC$ and $WUSC$.
\end{remark}

These scalings are natural conditions on $\psi(r)$ in the unimodal setting and there are many examples of L\'{e}vy-Khintchine exponents which satisfy $WLSC$ or $WUSC$. For example, as is shown in \cite{Bog1}, for any unimodal L\'{e}vy process we have
$$
\psi \in WLSC\left( 0,0, 1/\pi^{2} \right) \cap WUSC \left( 2,0,\pi^{2} \right).
$$

\noindent Another example is $\psi(\xi) = | \xi |^{\alpha}$, the L\'{e}vy-Khintchine exponent of the isotropic $\alpha$-stable L\'{e}vy process in $\mathbb{R}^{d}$ with $\alpha \in (0,2)$. This satisfies $WLSC(\alpha,0,1)$ and $WUSC(\alpha, 0,1)$. Alternatively, a non-stable example is $\psi(\xi) = |\xi|^{\alpha_{1}} + |\xi|^{\alpha_{2}}$, for which we have $\psi(\xi) \in WLSC \left( \alpha_{1}, 0,1 \right) \cap WUSC \left( \alpha_{2}, 0,1 \right)$, where $0 < \alpha_{1} < \alpha_{2} < 2$. Finally, if $\psi(r)$ is $\alpha$-regular varying at infinity and $0 < \alpha < 2$, then $\psi \in WLSC \left( \underline{\alpha} \right) \cap WUSC \left( \overline{\alpha} \right)$, for any $0 < \underline{\alpha} < \alpha < \overline{\alpha} < 2$. See \cite{Bog1} for more details on $WLSC$ and $WUSC$.

\*

\begin{remark}
By definition, if $\psi \in WLSC \left( \underline{\alpha}, \underline{\theta} \right)$, then there exists some constant $\underline{C}$ such that
$$ \frac{V \left( \frac{1}{\lambda r} \right)}{V \left(\frac{1}{r} \right)} \leq \underline{C} \lambda^{-\underline{\alpha}/2}, $$
for $\lambda \geq 1$ and $r > \underline{\theta}$. That is,
\begin{equation} \label{Eq4}
\frac{V \left( \varepsilon s \right)}{V \left( s \right)} \leq \underline{C} \varepsilon^{\underline{\alpha}/2},
\end{equation}
for $0 < \varepsilon \leq 1$ and $s < 1/\underline{\theta}$. Similarly, if $\psi \in WUSC \left(\overline{\alpha}, \overline{\theta} \right)$ then there exists some constant $\overline{C}$ such that
\begin{equation} \label{Eq5}
\frac{V \left( s \right)}{V \left( \varepsilon s \right)} \leq \overline{C} \varepsilon^{-\overline{\alpha}/2},
\end{equation}
for $0 < \varepsilon \leq 1$ and $s < 1/\overline{\theta}$.
\end{remark}

\*

\begin{lemma}[Potter-like Bound] \label{Lemma6}
If $\psi \in WLSC(\underline{\alpha}, \underline{\theta}) \cap WUSC(\overline{\alpha}, \overline{\theta})$, $0 < x < 1/\overline{\theta}$, and $0 < y < 1/\underline{\theta}$, then there exists some constant $C$ such that
\begin{equation} \label{Eq6}
\frac{V(x)}{V(y)} \leq C \left( \left( \frac{x}{y} \right)^{\underline{\alpha}/2} \vee \left( \frac{x}{y} \right)^{\overline{\alpha}/2} \right).
\end{equation}
\end{lemma}
\begin{proof}
Using (\ref{Eq4}) and (\ref{Eq5}) we have
\begin{eqnarray}
\frac{V(x)}{V(y)} & = &
\begin{cases}
\frac{V(t y)}{V(y)}, & \text{ if } t = \frac{x}{y} \leq 1, \\
\frac{V(x)}{V(t^{-1} x)}, & \text{ if } t^{-1} = \frac{y}{x} \leq 1.
\end{cases} \nonumber \\
& \leq & 
\begin{cases}
\underline{C} t^{\underline{\alpha}/2}, & \text{ if } t = \frac{x}{y} \leq 1 \text{ and } y < 1/\underline{\theta}, \\
\overline{C} t^{\overline{\alpha}/2}, & \text{ if } t^{-1} = \frac{y}{x} \leq 1 \text{ and } x < 1/\overline{\theta}.
\end{cases} \nonumber \\
& \leq & C \left( \left( \frac{x}{y} \right)^{\underline{\alpha}/2} \vee \left( \frac{x}{y} \right)^{\overline{\alpha}/2} \right), \hspace{0.2in} \text{ for } \hspace{0.2in} x < 1/\overline{\theta}, \hspace{0.1in}  y < \underline{\theta}. \nonumber
\end{eqnarray}
\end{proof}

\begin{note}
We heavily use the inverse function of $V(x)$ on $[0,\infty)$ in this paper. Thus we choose the notation
\begin{equation}
T(t) := V^{-1} \left( \sqrt{t} \right).
\end{equation}
\end{note}

\noindent This is equivalent to $V^{2} \left( T(t) \right) = t$. For example, $T(t) = t^{1/\alpha}$ for the isotropic $\alpha$-stable L\'{e}vy process. The scaling properties of $T(t)$ at zero reflect those of $\psi(\xi)$ at infinity. See \cite{Bog5} for further discussion of $T(t)$.

\*

Throughout the rest of this paper we will make the following assumptions:
\begin{itemize}
\item Our L\'{e}vy measure $\nu$ is unimodal and infinite on $\mathbb{R}^d$ with $d \geq 2$
\item Our L\'{e}vy-Khintchine exponent satisfies
$$ 0 \neq \psi \in WLSC(\underline{\alpha}, \theta) \cap WUSC(\overline{\alpha}, \theta), $$
for some constants $0 < \underline{\alpha} \leq \overline{\alpha} < 2$ and $0 \leq \theta \leq \inf_{x \in D} \left(1/\delta_{D}(x) \right)$.
\end{itemize}

\begin{note}
These assumptions guarantee that the Hartman-Wintner condition, mentioned above in (\ref{Hart}), is satisfied. It is also worth noting that many partial results below require less assumptions, but for simplicity of the presentation we ignore such extensions.
\end{note}

\*

\subsection{Heat Kernel} \* \vspace{0.1in}

Let $p_{t}(x - y) = p(t, x, y)$ denote the (smooth) transition density function associated to the distribution of our L\'{e}vy process, $X_t$, starting at the point $x$.

\begin{definition}
The {\bf first exit time} of $X$ from $D$ is defined by
\begin{equation}
\tau_{D} = \inf \left\{ t > 0 : X_t \notin D \right\}.
\end{equation}
\end{definition}

\begin{definition}
For $t > 0$ and $x,y \in \mathbb{R}^d$ the {\bf heat remainder} of $X_{t}$ is defined to be
\begin{equation}
r_{D}(t,x,y) = \mathbb{E}^{x} \left[ \tau_{D} < t, \ p_{t - \tau_{D}} \left(X \left( \tau_{D} \right) - y \right) \right]. \label{Rem}
\end{equation}
\end{definition}

\begin{definition}
The {\bf Dirichlet heat kernel} of $X_{t}$ is the transition density of the process killed upon exiting $D$ and is given by the {\bf Hunt formula}:
\begin{equation}
p_{D}(t, x, y) = p_{t} (y - x) - r_{D}(t, x, y). \label{HeatK}
\end{equation}
\end{definition}

\begin{definition}
The {\bf trace} of the heat kernel $p_{D} (t, x, x)$ is given by
\begin{equation}
Z_{D}(t) = \int_{\mathbb{R}^d} p_{D} (t,x,x)dx. \label{Trace}
\end{equation}
\end{definition}

\noindent Eventually we will refer to the Green function of $X$ on $D$ using the followng notation:
\begin{definition}
Let $M \geq 0$. The {\bf truncated Green function} of the process $X$ on $D$ is defined by
\begin{equation}
G_{D}^{M} (x,y) = \int_{0}^{M} p_{D} (t, x, y) dt. \label{Trunc}
\end{equation}
\end{definition}

\noindent We will also refer to the Poisson kernel using the following notation:
\begin{definition}
Let $M \geq 0$. The {\bf truncated Poisson kernel} of the process $X$ on $D$ is defined by
\begin{equation}
K_{D}^{M} (x,z) = \int_{D} G_{D}^{M}(x,y) \nu(y - z) dy.\label{Poiss}
\end{equation}
\end{definition}

\*

\section{Main Theorem}

Our main theorem coincides exactly with what would be predicted based on previous work in \cite{Bog5} and \cite{Ban2}.
\begin{theorem} \label{Main}
Let $D \subset \mathbb{R}^{d}$, $d \geq 2$, be a bounded Lipschitz domain. Let $|D|$ denote the $d$-dimensional Lebesgue measure of $D$, and let $\mathcal{H}^{d-1} \left( \partial D \right)$ denote the $(d-1)$-dimensional Hausdorff measure of $\partial D$. Given any unimodal L\'{e}vy process and any $\varepsilon > 0$, there exists a $t_{0} > 0$ such that for any $0 < t < t_{0}$ the trace of the heat kernel satisfies
\begin{equation}
\left| Z_{D}(t) - p_{t}(0) |D| + C_{\mathbb{H}}(t) \mathcal{H}^{d-1}(\partial D) \right| \leq c(\varepsilon) T(t)^{1-d}, \label{MainEq}
\end{equation}
where $c(\varepsilon) \to 0$ as $\varepsilon \to 0$, and
\begin{equation}
C_{\mathbb{H}} (t) = T(t)^{1-d} \int_{0}^{\infty} r_{\mathbb{H}} \left(t, \left(q, 0,..., 0 \right), \left(q, 0, ..., 0 \right) \right) dq.
\end{equation}
Here
\begin{equation}
\mathbb{H} = \left\{ \left(x_{1}, ..., x_{d} \right) \in \mathbb{R}^{d} : x_{1} > 0 \right\} = \mathbb{R}_{+}^{d}
\end{equation}
is the upper half-space of $\mathbb{R}^{d}$.
\end{theorem}

\*

\subsection{Domain} \* \vspace{0.1in}

Let $D$ be a bounded Lipschitz domain. In order to prove our theorem we treat our Lipschitz domain $D$, as it was treated in \cite{Ban2} and \cite{Brown}; by dividing it into good and bad sets.

\begin{definition}
Let $\varepsilon, r > 0$. We say that $G \subset \partial D$ is $(\varepsilon, r)$-{\bf good} if for each point $q \in G$ the unit inner normal, $v(q)$, exists and
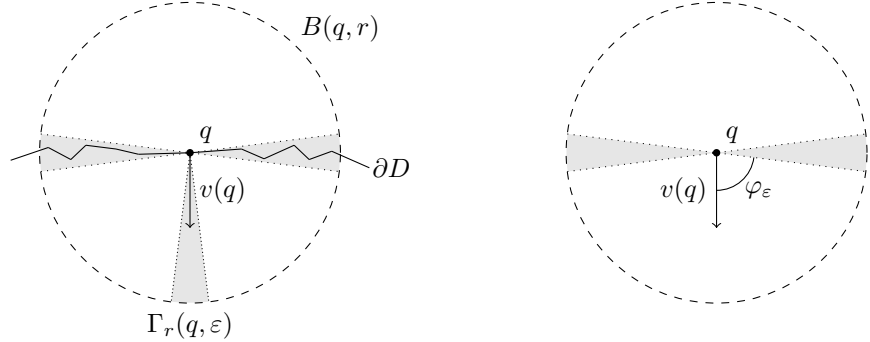
\begin{equation}
B(q,r) \cap \partial D \subset \left\{ x : |(x - q) \cdot v(q)| < \varepsilon |x - q| \right\}. \nonumber
\end{equation}
\end{definition}

\begin{center}
\begin{tikzpicture}
\fill[gray!20] (2,2) -- (1.75,0.02) arc (262.82:277.18:2) -- cycle;
\fill[gray!20] (2,2) -- (0.02,2.25) arc (172.82:187.18:2) -- cycle;
\fill[gray!20] (2,2) -- (3.98,1.75) arc (-7.18:7.18:2) -- cycle;
\draw[dashed]  (2,2) circle (2cm);
\draw (2,2) node[circle, fill, inner sep=1pt] {};
\draw (2,2) node[above right] {$q$};
\draw[dotted] ({2 - sqrt(4 - (1.75 - 2)^2)},1.75) -- ({2 + sqrt(4 - (2.25 - 2)^2)},2.25);
\draw[dotted] ({2 - sqrt(4 - (1.75 - 2)^2)},2.25) -- ({2 + sqrt(4 - (2.25 - 2)^2)},1.75);

\draw (2,2) -- ({2 + sqrt(4 - (2.25 - 2)^2) - 1.3}, 2.05);
\draw ({2 + sqrt(4 - (2.25 - 2)^2) - 1.3},2.05) -- ({2 + sqrt(4 - (2.25 - 2)^2) - 1},1.92);
\draw ({2 + sqrt(4 - (2.25 - 2)^2) - 1},1.92) -- ({2 + sqrt(4 - (2.25 - 2)^2) - 0.6},2.1);
\draw ({2 + sqrt(4 - (2.25 - 2)^2) - 0.6},2.1) -- ({2 + sqrt(4 - (2.25 - 2)^2) - 0.4},1.91);
\draw ({2 + sqrt(4 - (2.25 - 2)^2) - 0.4},1.91) -- ({2 + sqrt(4 - (2.25 - 2)^2) - 0.1},2.02);
\draw ({2 + sqrt(4 - (2.25 - 2)^2) - 0.1},2.02) -- ({2 + sqrt(4 - (2.25 - 2)^2) + 0.4},1.8);

\draw (2,2) -- ({2 - sqrt(4 - (2.25 - 2)^2) + 1.3}, 1.98);
\draw ({2 - sqrt(4 - (2.25 - 2)^2) + 1.3}, 1.98) -- ({2 - sqrt(4 - (2.25 - 2)^2) + 1},2.05);
\draw ({2 - sqrt(4 - (2.25 - 2)^2) + 1},2.05) -- ({2 - sqrt(4 - (2.25 - 2)^2) + 0.6},2.1);
\draw ({2 - sqrt(4 - (2.25 - 2)^2) + 0.6},2.1) -- ({2 - sqrt(4 - (2.25 - 2)^2) + 0.4},1.91);
\draw ({2 - sqrt(4 - (2.25 - 2)^2) + 0.4},1.91) -- ({2 - sqrt(4 - (2.25 - 2)^2) + 0.1},2.07);
\draw ({2 - sqrt(4 - (2.25 - 2)^2) + 0.1},2.07) -- ({2 - sqrt(4 - (2.25 - 2)^2) - 0.4},1.9);

\draw[->] (2,2) -- node[right]{$v (q)$} (2,1);

\draw[densely dotted] (2,2) -- (1.75, {2 - sqrt(4 - (1.75 - 2)^2)});
\draw[densely dotted] (2,2) -- (2.25, {2 - sqrt(4 - (1.75 - 2)^2)});

\draw (2,0) node[below] {$\Gamma_{r}(q,\varepsilon)$};

\draw (4.3,1.8) node[right] {$\partial D$};

\draw (3.35,3.35) node[above right] {$B(q,r)$};

\fill[gray!20] (9,2) -- (7.02,2.25) arc (172.82:187.18:2) -- cycle;
\fill[gray!20] (9,2) -- (10.98,1.75) arc (-7.18:7.18:2) -- cycle;

\draw[dashed]  (9,2) circle (2cm);
\draw (9,2) node[circle, fill, inner sep=1pt] {};
\draw (9,2) node[above right] {$q$};
\draw[dotted] ({9 - sqrt(4 - (1.75 - 2)^2)},1.75) -- ({9 + sqrt(4 - (2.25 - 2)^2)},2.25);
\draw[dotted] ({9 - sqrt(4 - (1.75 - 2)^2)},2.25) -- ({9 + sqrt(4 - (2.25 - 2)^2)},1.75);
\draw[->] (9,2) -- node[left]{$v (q)$} (9,1);
\draw (9,1.5) arc (270:352.818:0.5cm);
\draw (9.25,1.5) node[below, right]{$\varphi_{\varepsilon}$};
\end{tikzpicture}
\end{center}

\noindent Here $\varphi_{\varepsilon} \in [0,\pi/2]$ denotes the angle, measured from $v(q)$, such that $\cos \left( \varphi_{\varepsilon} \right) = \varepsilon$.

\begin{definition}
If $G$ is an $(\varepsilon, r)$-good set, then a {\bf good} subset, $\mathcal{G}$, of $D$ is a set of points of the form
\begin{equation}
\mathcal{G} := \bigcup_{q \in G} \Gamma_r (q,\varepsilon),
\end{equation}
where $\Gamma_{r}(q,\varepsilon)$ is a cone given by
\begin{equation}
\Gamma_{r}(q,\varepsilon) := \left\{ x : ( x - q ) \cdot v(q) > \sqrt{1 - \varepsilon^2} | x - q | \right\} \cap B(q,r). 
\end{equation}
\end{definition}

\noindent Let us define $\delta_{D}(x) := dist \left( x, \partial D \right)$, $x \in \mathbb{R}^{d}$. In \cite{Ban2}, the results Lemma~2.7 and Lemma~2.8 are combined to give the following result:
\begin{lemma}\label{Good}
Let $0< \varepsilon < 1/4$ and $r > 0$. There exists a measurable $(\varepsilon,r)$-good set $G \subset \partial D$ and $s_0 (\partial D, G)$ such that for all $s < s_0$
\begin{equation}
\left| \left\{ x \in D : \delta_{D}(x) < s \right\} \backslash \mathcal{G} \right| \leq s \varepsilon \left( 4 + \mathcal{H}^{d-1} \left( \partial D \right) \right). \label{Bad}
\end{equation}
\end{lemma}

\*

\subsection{Inner and Outer Cone} \* \vspace{0.1in}

Let $G \subset \partial D$ be an $(\varepsilon, r)$-good set and let $\mathcal{G}$ be good subset of $D$. If $x \in \mathcal{G}$, then, by definition, there exists a point $q(x) \in \partial D$ such that $x \in \Gamma_{r} (q(x), \varepsilon)$.

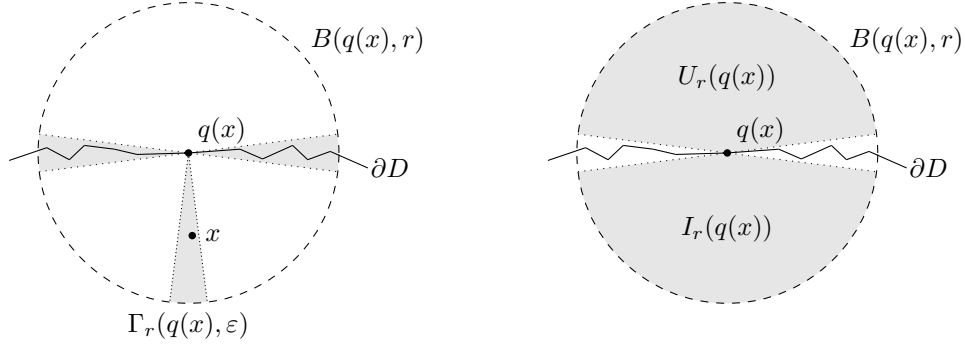
\begin{center}
\begin{tikzpicture}
\fill[gray!20] (2,2) -- (1.75,0.02) arc (262.82:277.18:2) -- cycle;
\fill[gray!20] (2,2) -- (0.02,2.25) arc (172.82:187.18:2) -- cycle;
\fill[gray!20] (2,2) -- (3.98,1.75) arc (-7.18:7.18:2) -- cycle;

\draw[dashed]  (2,2) circle (2cm);
\draw (2,2) node[circle, fill, inner sep=1pt] {};
\draw (2,2) node[above right] {$q(x)$};
\draw[dotted] ({2 - sqrt(4 - (1.75 - 2)^2)},1.75) -- ({2 + sqrt(4 - (2.25 - 2)^2)},2.25);
\draw[dotted] ({2 - sqrt(4 - (1.75 - 2)^2)},2.25) -- ({2 + sqrt(4 - (2.25 - 2)^2)},1.75);

\draw (2,2) -- ({2 + sqrt(4 - (2.25 - 2)^2) - 1.3}, 2.05);
\draw ({2 + sqrt(4 - (2.25 - 2)^2) - 1.3},2.05) -- ({2 + sqrt(4 - (2.25 - 2)^2) - 1},1.92);
\draw ({2 + sqrt(4 - (2.25 - 2)^2) - 1},1.92) -- ({2 + sqrt(4 - (2.25 - 2)^2) - 0.6},2.1);
\draw ({2 + sqrt(4 - (2.25 - 2)^2) - 0.6},2.1) -- ({2 + sqrt(4 - (2.25 - 2)^2) - 0.4},1.91);
\draw ({2 + sqrt(4 - (2.25 - 2)^2) - 0.4},1.91) -- ({2 + sqrt(4 - (2.25 - 2)^2) - 0.1},2.02);
\draw ({2 + sqrt(4 - (2.25 - 2)^2) - 0.1},2.02) -- ({2 + sqrt(4 - (2.25 - 2)^2) + 0.4},1.8);

\draw (2,2) -- ({2 - sqrt(4 - (2.25 - 2)^2) + 1.3}, 1.98);
\draw ({2 - sqrt(4 - (2.25 - 2)^2) + 1.3}, 1.98) -- ({2 - sqrt(4 - (2.25 - 2)^2) + 1},2.05);
\draw ({2 - sqrt(4 - (2.25 - 2)^2) + 1},2.05) -- ({2 - sqrt(4 - (2.25 - 2)^2) + 0.6},2.1);
\draw ({2 - sqrt(4 - (2.25 - 2)^2) + 0.6},2.1) -- ({2 - sqrt(4 - (2.25 - 2)^2) + 0.4},1.91);
\draw ({2 - sqrt(4 - (2.25 - 2)^2) + 0.4},1.91) -- ({2 - sqrt(4 - (2.25 - 2)^2) + 0.1},2.07);
\draw ({2 - sqrt(4 - (2.25 - 2)^2) + 0.1},2.07) -- ({2 - sqrt(4 - (2.25 - 2)^2) - 0.4},1.9);

\draw (2.05,0.9) node[circle, fill, inner sep=1pt] {};
\draw (2.1,0.9) node[right] {$x$};

\draw[densely dotted] (2,2) -- (1.75, {2 - sqrt(4 - (1.75 - 2)^2)});
\draw[densely dotted] (2,2) -- (2.25, {2 - sqrt(4 - (1.75 - 2)^2)});

\draw (2,0) node[below] {$\Gamma_{r}(q(x),\varepsilon)$};

\draw (4.3,1.8) node[right] {$\partial D$};

\draw (3.5,3.2) node[above right] {$B(q(x),r)$};
\end{tikzpicture}
\hspace{0.5in}
\begin{tikzpicture}
\fill[gray!20] (2,2) -- (3.984313,2.25) arc (7.125:172.875:2) -- cycle;
\fill[gray!20] (2,2) -- (0.0156865,1.75) arc (187.125:352.875:2) -- cycle;

\draw[dashed]  (2,2) circle (2cm);
\draw (2,2) node[circle, fill, inner sep=1pt] {};
\draw (2,2) node[above right] {$q(x)$};
\draw[dotted] ({2 - sqrt(4 - (1.75 - 2)^2)},1.75) -- ({2 + sqrt(4 - (2.25 - 2)^2)},2.25);
\draw[dotted] ({2 - sqrt(4 - (1.75 - 2)^2)},2.25) -- ({2 + sqrt(4 - (2.25 - 2)^2)},1.75);

\draw (2,2) -- ({2 + sqrt(4 - (2.25 - 2)^2) - 1.3}, 2.05);
\draw ({2 + sqrt(4 - (2.25 - 2)^2) - 1.3},2.05) -- ({2 + sqrt(4 - (2.25 - 2)^2) - 1},1.92);
\draw ({2 + sqrt(4 - (2.25 - 2)^2) - 1},1.92) -- ({2 + sqrt(4 - (2.25 - 2)^2) - 0.6},2.1);
\draw ({2 + sqrt(4 - (2.25 - 2)^2) - 0.6},2.1) -- ({2 + sqrt(4 - (2.25 - 2)^2) - 0.4},1.91);
\draw ({2 + sqrt(4 - (2.25 - 2)^2) - 0.4},1.91) -- ({2 + sqrt(4 - (2.25 - 2)^2) - 0.1},2.02);
\draw ({2 + sqrt(4 - (2.25 - 2)^2) - 0.1},2.02) -- ({2 + sqrt(4 - (2.25 - 2)^2) + 0.4},1.8);

\draw (2,2) -- ({2 - sqrt(4 - (2.25 - 2)^2) + 1.3}, 1.98);
\draw ({2 - sqrt(4 - (2.25 - 2)^2) + 1.3}, 1.98) -- ({2 - sqrt(4 - (2.25 - 2)^2) + 1},2.05);
\draw ({2 - sqrt(4 - (2.25 - 2)^2) + 1},2.05) -- ({2 - sqrt(4 - (2.25 - 2)^2) + 0.6},2.1);
\draw ({2 - sqrt(4 - (2.25 - 2)^2) + 0.6},2.1) -- ({2 - sqrt(4 - (2.25 - 2)^2) + 0.4},1.91);
\draw ({2 - sqrt(4 - (2.25 - 2)^2) + 0.4},1.91) -- ({2 - sqrt(4 - (2.25 - 2)^2) + 0.1},2.07);
\draw ({2 - sqrt(4 - (2.25 - 2)^2) + 0.1},2.07) -- ({2 - sqrt(4 - (2.25 - 2)^2) - 0.4},1.9);

\draw (2,1) node {$I_{r}(q(x))$};
\draw (2,3) node {$U_{r}(q(x))$};

\draw (3.5,3.2) node[above right] {$B(q(x),r)$};

\draw (4.3,1.8) node[right] {$\partial D$};

\draw (2,0) node[below] {\phantom{$\Gamma_{r}(p(x),\varepsilon)$}};

\end{tikzpicture}
\end{center}

\noindent We define the {\bf Inner} and {\bf Outer} cones of $B(q(x), r)$ as follows
\begin{eqnarray}
I_{r}(q(x)) & := & \left\{ y \ : \ (y - q(x)) \cdot v(q(x)) >  \varepsilon |y - q(x) | \right\} \cap B(q(x),r), \\
U_{r}(q(x)) & := & \left\{ y \ : \ (y - q(x)) \cdot v(q(x)) < - \varepsilon |y - q(x) | \right\} \cap B(q(x),r).
\end{eqnarray}

\noindent Note, for $x \in \mathcal{G}$, we have $$\Gamma_r (q(x), \varepsilon) \subset I_{r}(q(x)) \subset D \subset U_r^{c} (q(x)).$$

\noindent It is shown in \cite{Ban2} that for any $x \in \mathcal{G}$ there exists a half-space $H^{*} (x)$ such that:
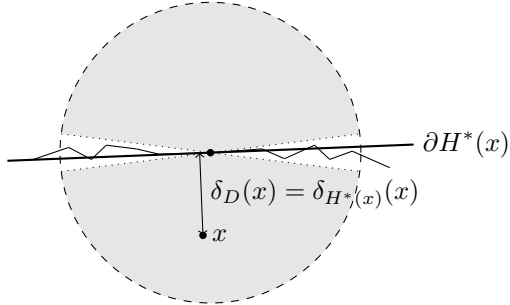
\begin{equation}
x \in H^{*} (x), \hspace{0.5in} \delta_{H^{*}(x)} (x) = \delta_{D} (x) \hspace{0.5in} I_{r}(q(x)) \subseteq H^{*}(x) \subseteq U^{c}_{r}(q(x)). 
\end{equation}

\begin{center}
\begin{tikzpicture}
\fill[gray!20] (2,2) -- (3.984313,2.25) arc (7.125:172.875:2) -- cycle;
\fill[gray!20] (2,2) -- (0.0156865,1.75) arc (187.125:352.875:2) -- cycle;

\draw[dashed]  (2,2) circle (2cm);
\draw (2,2) node[circle, fill, inner sep=1pt] {};
\draw[dotted] ({2 - sqrt(4 - (1.75 - 2)^2)},1.75) -- ({2 + sqrt(4 - (2.25 - 2)^2)},2.25);
\draw[dotted] ({2 - sqrt(4 - (1.75 - 2)^2)},2.25) -- ({2 + sqrt(4 - (2.25 - 2)^2)},1.75);

\draw (2,2) -- ({2 + sqrt(4 - (2.25 - 2)^2) - 1.3}, 2.05);
\draw ({2 + sqrt(4 - (2.25 - 2)^2) - 1.3},2.05) -- ({2 + sqrt(4 - (2.25 - 2)^2) - 1},1.92);
\draw ({2 + sqrt(4 - (2.25 - 2)^2) - 1},1.92) -- ({2 + sqrt(4 - (2.25 - 2)^2) - 0.6},2.1);
\draw ({2 + sqrt(4 - (2.25 - 2)^2) - 0.6},2.1) -- ({2 + sqrt(4 - (2.25 - 2)^2) - 0.4},1.91);
\draw ({2 + sqrt(4 - (2.25 - 2)^2) - 0.4},1.91) -- ({2 + sqrt(4 - (2.25 - 2)^2) - 0.1},2.02);
\draw ({2 + sqrt(4 - (2.25 - 2)^2) - 0.1},2.02) -- ({2 + sqrt(4 - (2.25 - 2)^2) + 0.4},1.8);

\draw (2,2) -- ({2 - sqrt(4 - (2.25 - 2)^2) + 1.3}, 1.98);
\draw ({2 - sqrt(4 - (2.25 - 2)^2) + 1.3}, 1.98) -- ({2 - sqrt(4 - (2.25 - 2)^2) + 1},2.05);
\draw ({2 - sqrt(4 - (2.25 - 2)^2) + 1},2.05) -- ({2 - sqrt(4 - (2.25 - 2)^2) + 0.6},2.1);
\draw ({2 - sqrt(4 - (2.25 - 2)^2) + 0.6},2.1) -- ({2 - sqrt(4 - (2.25 - 2)^2) + 0.4},1.91);
\draw ({2 - sqrt(4 - (2.25 - 2)^2) + 0.4},1.91) -- ({2 - sqrt(4 - (2.25 - 2)^2) + 0.1},2.07);
\draw ({2 - sqrt(4 - (2.25 - 2)^2) + 0.1},2.07) -- ({2 - sqrt(4 - (2.25 - 2)^2) - 0.4},1.9);

\draw (1.9,0.9) node[circle, fill, inner sep=1pt] {};
\draw (1.9,0.9) node[right] {$x$};

\draw[thick] (-0.7, {0.04*(-0.7 - 2) + 2}) -- (4.7, {0.04*(4.7 - 2) + 2});

\draw (4.7,{0.04*(4.7 - 2) + 2}) node[right] {$\partial H^{*}(x)$};

\draw[<->] (1.9,0.9) -- node[right] {$\delta_{D}(x) = \delta_{H^{*}(x)} (x)$} (1.856, {(-1/0.04)*(1.856 - 1.9) + 0.9}); 

\end{tikzpicture}
\end{center}

\*

\section{Proof of the Main Theorem} \* \vspace{0.1in}

The transition densities of isotropic processes killed upon exiting a domain $D$ are given by the Hunt formula
\begin{equation}
p_{D}(t, x, y) = p_{t} (y - x) - r_{D}(t, x, y).
\end{equation}

\noindent It follows that
\begin{eqnarray}
-\int_{D} r_{D} (t, x, x) dx & = & \int_{D} p_{D}(t, x, x) \hspace{0.02in} dx - \int_{D} p_{t} (0) \hspace{0.02in} dx \nonumber \\
& = & Z_{D}(t) - p_{t}(0) |D|.
\end{eqnarray}

\noindent Hence in order to prove Theorem~\ref{Main} it is sufficient to show that for an arbitrary $\varepsilon > 0$ there exists a $t_0 > 0$ such that for any $0 < t < t_{0}$ we have
\begin{equation}
\left| \int_{D} r_{D} (t, x, x) \hspace{0.02in} dx - C_{\mathbb{H}}(t) \mathcal{H}^{d-1}(\partial D) \right| \leq c(\varepsilon) T(t)^{1-d},
\end{equation}
where $c(\varepsilon) \to 0$ as $\varepsilon \to 0$.

\noindent We need to estimate $$ \int_{D} r_{D}(t, x, x) \hspace{0.02in} dx.$$

\noindent Fix $0 < \varepsilon <1/4$. Let us define $G \subset \partial D$ to be the $(\varepsilon,r)$-good set as described above in Lemma~\ref{Good}. Let $\mathcal{G}$ be the corresponding good subset of $D$. Then we divide $D$ into the following domains
\begin{eqnarray}
D_1 & = & \left\{ x \in D \backslash \mathcal{G} \ : \ \delta_{D} (x) < s \right\}, \nonumber \\
D_2 & = & \left\{ x \in D \cap \mathcal{G} \ : \ \delta_{D} (x) < s \right\}, \nonumber \\
D_3 & = & \left\{ x \in D \ : \ \delta_{D} (x) \geq s \right\}, \nonumber 
\end{eqnarray}
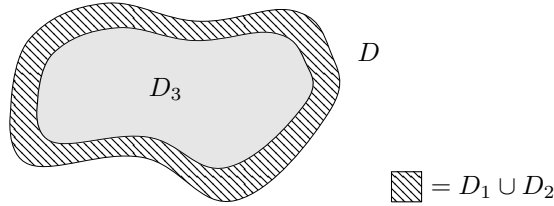
\begin{center}
\resizebox{1.85in}{1.1in}{
\includegraphics{myfig.tikz}
}

\vspace{-0.98in}

\resizebox{1.55in}{0.78in}{
\includegraphics{myfig1.tikz}
}

\vspace{-0.75in} \hspace{1.8in}
\begin{tikzpicture}
	\draw (2,2) node {$D_3$};
	\draw[pattern = north west lines] (5,0.5) -- (5.4,0.5) -- (5.4,0.9) -- (5,0.9) -- (5,0.5);
	\draw (5.4,0.7) node[right] {$= D_{1} \cup D_{2}$};
	\draw (4.7, 2.5) node{$D$};
\end{tikzpicture}
\end{center}
where $s$ must be smaller than the $s_{0}$ given in Lemma \ref{Good}. For small enough $t$ we can let $s = T(t)/\sqrt{\varepsilon}$.



%
%

\*

\subsection{The domain $D_{1}$:} \* \vspace{0.1in}

The following estimate for $r_{D} (t, x, y)$ comes from Lemma~2.4 of \cite{Bog5}.
\begin{lemma} \label{Lemma3}
Suppose $\psi \in WLSC(\underline{\alpha}, \theta)$ and $T(t) < 1/\theta$. Then
\begin{eqnarray}
r_{D} (t, x, y) \leq C \left\{T(t)^{-d} \wedge \frac{t}{\delta_{D}^{d}(x) V^{2} \left( \delta_{D}(x) \right)} \right\}.
\end{eqnarray}
\end{lemma}

\noindent By assumption $\psi \in WLSC \left(\underline{\alpha}, \theta \right)$ and so, for us, this lemma implies that if $T(t) < 1/\theta$, then
\begin{eqnarray}
\int_{D_{1}} r_{D} (t, x, x) \hspace{0.02in} dx & \leq & C \int_{D_{1}} T(t)^{-d} \hspace{0.02in} dx \\
& = & C T(t)^{-d} |D_{1}|.
\end{eqnarray}

\noindent But, by Lemma~\ref{Good}, we know that the measure of the set of bad points near the boundary is small. Hence if $T(t) < 1/\theta$, then
\begin{equation}
\int_{D_{1}} r_{D} (t, x, x) \hspace{0.02in} dx \leq C(\partial D) \varepsilon s T(t)^{-d} \leq \boxed{C \sqrt{\varepsilon} T(t)^{1-d}}, \label{EqA}
\end{equation}
where $C$ is a constant depending on $d$, $\underline{\alpha}$, and $\partial D$.

\*

\subsection{The domain $D_{3}$:} \* \vspace{0.1in}

By assumption $\psi \in WLSC(\underline{\alpha}, \theta)$, and so, again by Lemma~\ref{Lemma3}, if $T(t) < 1/\theta$, then
\begin{equation}
\int_{D_{3}} r_{D} (t, x, x) \hspace{0.02in} dx \leq C T(t)^{-d} \int_{D_3} \left\{ 1 \wedge \frac{T(t)^{d}}{\delta^d_{D} (x)} \frac{V^2(T(t))}{V^2(\delta_{D}(x))} \right\} dx.
\end{equation}

\noindent Next, our Potter-like bound in Lemma~\ref{Lemma6} tells us that if $T(t) < 1/\theta$, then
\begin{equation}
\int_{D_{3}} r_{D} (t, x, x) \hspace{0.02in} dx \leq C T(t)^{-d} \int_{D_3} \left\{ 1 \wedge  \frac{T(t)^{d}}{\delta^d_{D} (x)} \left( \frac{T(t)^{\underline{\alpha}}}{\delta_{D}^{\underline{\alpha}} (x)} \vee \frac{T(t)^{\overline{\alpha}}}{\delta_{D}^{\overline{\alpha}} (x)} \right) \right\} dx.
\end{equation}

\noindent By definition of $D_{3}$, for any $x \in D_{3}$ we have $\delta_{D}(x) \geq s = T(t)/\sqrt{\varepsilon}$. Or equivalently $1 \leq \frac{\delta_{D}(x)}{T(t)} \sqrt{\varepsilon}$. Hence
\begin{eqnarray}
\int_{D_{3}} r_{D} (t, x, x) \hspace{0.02in} dx & \leq & C T(t)^{-d} \int_{D_3} \left\{ 1 \wedge  \sqrt{\varepsilon} \frac{T(t)^{d-1}}{\delta^{d-1}_{D} (x)} \left( \frac{T(t)^{\underline{\alpha}}}{\delta_{D}^{\underline{\alpha}} (x)} \vee \frac{T(t)^{\overline{\alpha}}}{\delta_{D}^{\overline{\alpha}} (x)} \right) \right\} dx \\
& \leq & C T(t)^{-d} \int_{D} \left\{ 1 \wedge \sqrt{\varepsilon} \frac{T(t)^{d + \underline{\alpha} - 1}}{\delta_{D}^{d + \underline{\alpha} - 1}(x)} + 1 \wedge \sqrt{\varepsilon} \frac{T(t)^{d + \overline{\alpha} - 1}}{\delta_{D}^{d + \overline{\alpha} - 1}(x)} \right\} dx \\
& = & C T(t)^{1-d} \frac{1}{T(t)} \int_{D} \left\{ 1 \wedge \sqrt{\varepsilon} \left( \frac{\delta_{D}(x)}{T(t)} \right)^{-d - \underline{\alpha} + 1} + 1 \wedge \sqrt{\varepsilon} \left( \frac{\delta_{D}(x)}{T(t)} \right)^{-d - \overline{\alpha} + 1} \right\} dx. \label{Integ}
\end{eqnarray} 

We are now in a position to apply the following important proposition from \cite{Ban2}:
\begin{proposition} \label{Prop1}
Let $D \subset \mathbb{R}^d$ be a bounded Lipschitz domain. Suppose that $f : (0,\infty) \rightarrow \mathbb{R}$ is continuous and satisfies $f(s) \leq c \left( 1 \wedge s^{-\beta} \right)$, $s > 0$, for some $\beta > 1$, and suppose that for any $0 < R_1 < R_2 < \infty$, $f(s)$ is Lipschitz on $\left[ R_1, R_2 \right]$. Then
\begin{equation}
\lim_{\eta \to 0^{+}} \frac{1}{\eta} \int_{D} f\left(\frac{\delta_{D}(x)}{\eta} \right) dx = \mathcal{H}^{d-1} (\partial D) \int_{0}^\infty f(s) \hspace{0.02in} ds.
\end{equation}
\end{proposition}

\noindent Letting $\eta = T(t)$ and $f(s) = 1 \wedge \sqrt{\varepsilon} s^{-d-\underline{\alpha}+1}$ and $f(s) = 1 \wedge \sqrt{\varepsilon} s^{-d-\overline{\alpha}+1}$, respectively, we can apply Proposition~\ref{Prop1} to both of the integrals in (\ref{Integ}). Thus for small values of $t$ we get
\begin{eqnarray}
\int_{D_{3}} r_{D} (t, x, x) \hspace{0.02in} dx & \leq & C T(t)^{1 - d}\mathcal{H}^{d - 1} (\partial D) \int_{0}^{\infty} \left\{ \left( 1 \wedge \sqrt{\varepsilon} r^{-d - \underline{\alpha} + 1} \right) + \left( 1 \wedge \sqrt{\varepsilon} r^{-d-\overline{\alpha}+1} \right) \right\} dr.
\end{eqnarray}

\noindent Using substitution this becomes
\begin{eqnarray}
\int_{D_{3}} r_{D} (t, x, x) \hspace{0.02in} dx & \leq & C(\partial D) T(t)^{1 - d} \left\{ \varepsilon^{\frac{1}{2(d + \underline{\alpha}-1)}} \int_{0}^{\infty} \left( 1 \wedge r^{-d - \underline{\alpha} + 1} \right) dr + \varepsilon^{\frac{1}{2(d + \overline{\alpha}-1)}} \int_{0}^{\infty} \left( 1 \wedge r^{-d-\overline{\alpha}+1} \right) dr \right\} \nonumber \\
& \leq & \boxed{C T(t)^{1-d} \left( \varepsilon^{\frac{1}{2(d + \underline{\alpha}-1)}} + \varepsilon^{\frac{1}{2(d + \overline{\alpha}-1)}} \right) }. \label{EqB}
\end{eqnarray}

\noindent This covers domains $D_1$ and $D_3$.

\*

\subsection{The domain $D_{2}$:} \* \vspace{0.1in}

It remains to show that $r_{D}(t, x, x)$ is comparable to $r_{H^{*}}(t, x, x)$ for $x \in D_{2}$. 

\noindent Suppose $x \in D_{2} \subset \mathcal{G}$. Let $q(x)$ be as above. Then $x \in \Gamma_{r} \left( q(x), \varepsilon \right)$. For the purposes of brevity we will use the folowing notation $\mathcal{I} := I_r(q(x))$ and $\mathcal{U}^{c} := U^c_r(q(x))$.
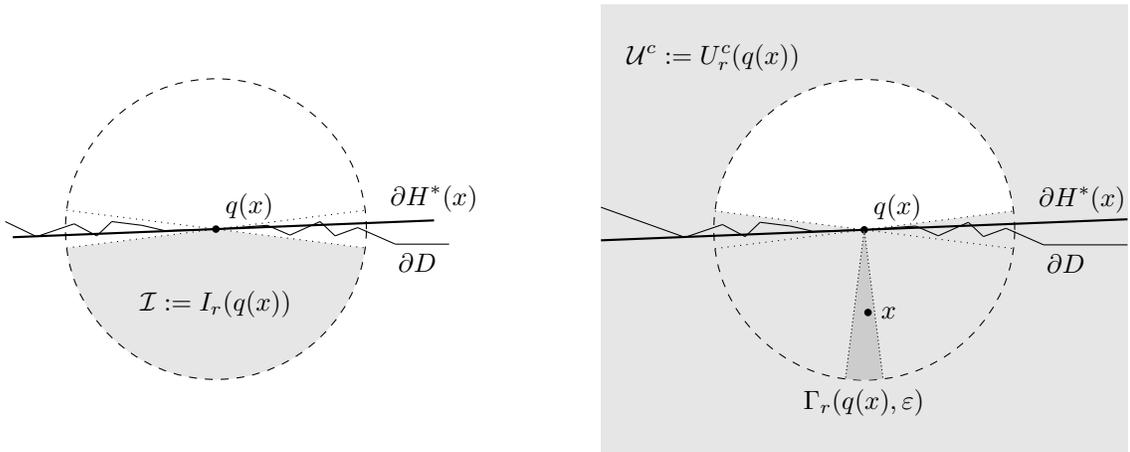
\begin{center}
\begin{tikzpicture}
\fill[gray!20] (2,3) -- (0.0156865,2.75) arc (187.125:352.875:2) -- cycle;

\draw[dashed]  (2,3) circle (2cm);
\draw (2,3) node[circle, fill, inner sep=1pt] {};
\draw (2,3) node[above right] {$q(x)$};
\draw[dotted] ({2 - sqrt(4 - (1.75 - 2)^2)},2.75) -- ({2 + sqrt(4 - (2.25 - 2)^2)},3.25);
\draw[dotted] ({2 - sqrt(4 - (1.75 - 2)^2)},3.25) -- ({2 + sqrt(4 - (2.25 - 2)^2)},2.75);

\draw (2,3) -- ({2 + sqrt(4 - (2.25 - 2)^2) - 1.3}, 3.05);
\draw ({2 + sqrt(4 - (2.25 - 2)^2) - 1.3},3.05) -- ({2 + sqrt(4 - (2.25 - 2)^2) - 1},2.92);
\draw ({2 + sqrt(4 - (2.25 - 2)^2) - 1},2.92) -- ({2 + sqrt(4 - (2.25 - 2)^2) - 0.6},3.1);
\draw ({2 + sqrt(4 - (2.25 - 2)^2) - 0.6},3.1) -- ({2 + sqrt(4 - (2.25 - 2)^2) - 0.4},2.91);
\draw ({2 + sqrt(4 - (2.25 - 2)^2) - 0.4},2.91) -- ({2 + sqrt(4 - (2.25 - 2)^2) - 0.1},3.02);
\draw ({2 + sqrt(4 - (2.25 - 2)^2) - 0.1},3.02) -- ({2 + sqrt(4 - (2.25 - 2)^2) + 0.4},2.8);
\draw ({2 + sqrt(4 - (2.25 - 2)^2) + 0.4},2.8) -- (5.1,2.8);

\draw (2,3) -- ({2 - sqrt(4 - (2.25 - 2)^2) + 1.3}, 2.98);
\draw ({2 - sqrt(4 - (2.25 - 2)^2) + 1.3}, 2.98) -- ({2 - sqrt(4 - (2.25 - 2)^2) + 1},3.05);
\draw ({2 - sqrt(4 - (2.25 - 2)^2) + 1},3.05) -- ({2 - sqrt(4 - (2.25 - 2)^2) + 0.6},3.1);
\draw ({2 - sqrt(4 - (2.25 - 2)^2) + 0.6},3.1) -- ({2 - sqrt(4 - (2.25 - 2)^2) + 0.4},2.91);
\draw ({2 - sqrt(4 - (2.25 - 2)^2) + 0.4},2.91) -- ({2 - sqrt(4 - (2.25 - 2)^2) + 0.1},3.07);
\draw ({2 - sqrt(4 - (2.25 - 2)^2) + 0.1},3.07) -- ({2 - sqrt(4 - (2.25 - 2)^2) - 0.4},2.9);
\draw ({2 - sqrt(4 - (2.25 - 2)^2) - 0.4},2.9) -- (-0.8,3.1);

\draw (2,2) node {$\mathcal{I} := I_{r} (q(x))$};

\draw (4.3,2.8) node[below right] {$\partial D$};

\draw (4.9,{0.04*(4.7 - 2) + 3}) node[above] {$\partial H^{*}(x)$};
\draw[thick] (-0.7, {0.04*(-0.7 - 2) + 3}) -- (4.9, {0.04*(4.9 - 2) + 3});

\draw (0, -0.0);

\end{tikzpicture} \hspace{0.5in}
\begin{tikzpicture}
\fill[gray!20] (-1.5,-1) -- (5.5,-1) -- (5.5,5) -- (-1.5,5);
\fill[white] (2,2) -- (3.984313,2.25) arc (7.125:172.875:2) -- cycle;

\draw[dashed]  (2,2) circle (2cm);
\draw (2,2) node[circle, fill, inner sep=1pt] {};
\draw (2,2) node[above right] {$q(x)$};
\draw[dotted] ({2 - sqrt(4 - (1.75 - 2)^2)},1.75) -- ({2 + sqrt(4 - (2.25 - 2)^2)},2.25);
\draw[dotted] ({2 - sqrt(4 - (1.75 - 2)^2)},2.25) -- ({2 + sqrt(4 - (2.25 - 2)^2)},1.75);

\draw (2,2) -- ({2 + sqrt(4 - (2.25 - 2)^2) - 1.3}, 2.05);
\draw ({2 + sqrt(4 - (2.25 - 2)^2) - 1.3},2.05) -- ({2 + sqrt(4 - (2.25 - 2)^2) - 1},1.92);
\draw ({2 + sqrt(4 - (2.25 - 2)^2) - 1},1.92) -- ({2 + sqrt(4 - (2.25 - 2)^2) - 0.6},2.1);
\draw ({2 + sqrt(4 - (2.25 - 2)^2) - 0.6},2.1) -- ({2 + sqrt(4 - (2.25 - 2)^2) - 0.4},1.91);
\draw ({2 + sqrt(4 - (2.25 - 2)^2) - 0.4},1.91) -- ({2 + sqrt(4 - (2.25 - 2)^2) - 0.1},2.02);
\draw ({2 + sqrt(4 - (2.25 - 2)^2) - 0.1},2.02) -- ({2 + sqrt(4 - (2.25 - 2)^2) + 0.4},1.8);
\draw ({2 + sqrt(4 - (2.25 - 2)^2) + 0.4},1.8) -- (5.5,1.8);

\draw (2,2) -- ({2 - sqrt(4 - (2.25 - 2)^2) + 1.3}, 1.98);
\draw ({2 - sqrt(4 - (2.25 - 2)^2) + 1.3}, 1.98) -- ({2 - sqrt(4 - (2.25 - 2)^2) + 1},2.05);
\draw ({2 - sqrt(4 - (2.25 - 2)^2) + 1},2.05) -- ({2 - sqrt(4 - (2.25 - 2)^2) + 0.6},2.1);
\draw ({2 - sqrt(4 - (2.25 - 2)^2) + 0.6},2.1) -- ({2 - sqrt(4 - (2.25 - 2)^2) + 0.4},1.91);
\draw ({2 - sqrt(4 - (2.25 - 2)^2) + 0.4},1.91) -- ({2 - sqrt(4 - (2.25 - 2)^2) + 0.1},2.07);
\draw ({2 - sqrt(4 - (2.25 - 2)^2) + 0.1},2.07) -- ({2 - sqrt(4 - (2.25 - 2)^2) - 0.4},1.9);
\draw ({2 - sqrt(4 - (2.25 - 2)^2) - 0.4},1.9) -- (-1.5,2.3);

\draw (0,4.3) node {$\mathcal{U}^{c} := U^{c}_{r} (q(x))$};

\draw (4.3,1.8) node[below right] {$\partial D$};

\draw (4.9,{0.04*(4.7 - 2) + 2}) node[above] {$\partial H^{*}(x)$};
\draw[thick] (-1.5, {0.04*(-1.5 - 2) + 2}) -- (5.5, {0.04*(5.5 - 2) + 2});

\fill[gray!40] (2,2) -- (1.75,0.02) arc (262.82:277.18:2) -- cycle;

\draw[densely dotted] (2,2) -- (1.75, {2 - sqrt(4 - (1.75 - 2)^2)});
\draw[densely dotted] (2,2) -- (2.25, {2 - sqrt(4 - (1.75 - 2)^2)});

\draw (2,0) node[below] {$\Gamma_{r}(q(x),\varepsilon)$};

\draw (2.05,0.9) node[circle, fill, inner sep=1pt] {};
\draw (2.1,0.9) node[right] {$x$};

\end{tikzpicture}
\end{center}

Notice that $$H^{*}(x) \subseteq \mathcal{U}^{c} \hspace{0.2in} \text{ and } \hspace{0.2in} \mathcal{I} \subseteq D.$$

\noindent Hence
\begin{eqnarray}
\left| r_{D} (t, x, x) - r_{H^{*}(x)} (t, x, x)  \right| & \leq & r_{\mathcal{I}} (t, x, x) - r_{\mathcal{U}^{c}} (t, x, x). \label{HuntPres}
\end{eqnarray}


\noindent We have the following important proposition:
\begin{proposition} \label{MainProp}
Let $v(q) \in \mathbb{R}^d$ be a unit vector. Assume that $0 < \varepsilon < 1/4$ and $r > 0$. If $x \in \Gamma_{2s}(v(q),\varepsilon)$ and $s = T(t)/\sqrt{\varepsilon} < r/4$, then
\begin{equation}
0 \leq r_{\mathcal{I}} (t, x, x) - r_{\mathcal{U}^{c}} (t, x, x) \leq \frac{\left(\varepsilon^{1 - \underline{\alpha}/2} + \varepsilon^{1 - \overline{\alpha}/2} \right) \vee \sqrt{\varepsilon}}{T(t)^{d}} \left( 1 \wedge \frac{T(t)^{d-1}}{\delta_{\mathcal{I}}^{d-1} (x)}  \frac{V^{2}(T(t))}{V^2\left( \delta_{\mathcal{I}} (x) \right)} \right).
\end{equation}
\end{proposition}
\noindent We postpone the proof of this proposition until Section~\ref{Sec5}. 

\noindent Using (\ref{HuntPres}) and Proposition~\ref{MainProp} we get
\begin{eqnarray}
 \int_{D_{2}} \left| r_{D} (t, x, x) - r_{H^{*}} (t, x, x) \right| dx & \leq & \int_{D_{2}} \left( r_{\mathcal{I}} (t, x, x) - r_{\mathcal{U}^{c}} (t, x, x) \right) dx \\
& \leq &\frac{ C\left(\varepsilon \right)}{T(t)^{d}} \int_{D_{2}} \left(1 \wedge  \frac{T(t)^{d-1}}{\delta_{\mathcal{I}}^{d-1} (x)}  \frac{V^{2}(T(t))}{V^2\left( \delta_{\mathcal{I}} (x) \right)} \right) dx.
\end{eqnarray}

\noindent Notice that since $x \in \Gamma_{2s}(v(q),\varepsilon)$, $\partial D \cap B(q,r) \subset B(q,r) \backslash \mathcal{I}$, and  $\varepsilon < 1/4$ we have
\begin{eqnarray}
\delta_{\mathcal{I}}(x) & \geq & | x - q | \sin \left( 2\varphi_{\varepsilon} - \pi / 2 \right) = - | x - q | \cos \left( 2\varphi_{\varepsilon} \right) \nonumber \\
& = & \left(1 - 2\varepsilon^{2} \right) |x - q| \geq \left(1 - 2\varepsilon^{2} \right) \delta_{D}(x) > \frac{7}{8} \delta_{D}(x).
\end{eqnarray}

\noindent Hence
\begin{equation}
\int_{D_{2}} \left| r_{D} (t, x, x) - r_{H^{*}} (t, x, x) \right| dx \leq \frac{C(\varepsilon)}{T(t)^{d}} \int_{D_{2}} \left(1 \wedge  \frac{T(t)^{d-1}}{\delta_{D}^{d-1} (x)}  \frac{V^{2}(T(t))}{V^2\left( \delta_{D} (x) \right)} \right) dx.
\end{equation}

\noindent We can use our Potter-like bounds from Lemma~\ref{Lemma6} again: if $T(t) < 1/\theta$, then
\begin{eqnarray}
\int_{D_{2}} \left| r_{D} (t, x, x) - r_{H^{*}} (t, x, x) \right| dx & \leq & \frac{C(\varepsilon)}{T(t)^{d}} \int_{D_{2}} \left\{1 \wedge  \frac{T(t)^{d-1}}{\delta_{D}^{d-1} (x)}  \left( \frac{T(t)^{\underline{\alpha}}}{\delta_{D}^{\underline{\alpha}} (x)} \vee \frac{T(t)^{\overline{\alpha}}}{\delta_{D}^{\overline{\alpha}} (x)} \right) \right\} dx  \\
& \leq & \frac{C(\varepsilon)}{T(t)^{d}} \int_{D_{2}} \left\{ 1 \wedge \left( \frac{T(t)}{\delta_{D} (x)} \right)^{d + \underline{\alpha} - 1} + 1 \wedge  \left( \frac{T(t)}{\delta_{D} (x)} \right)^{d + \overline{\alpha} - 1} \right\} dx. 
\end{eqnarray}

\noindent Letting $\eta = T(t)$ as above, we can apply Proposition~\ref{Prop1} to get, for small enough $t$, that
\begin{eqnarray}
 \int_{D_{2}} \left| r_{D} (t, x, x) - r_{H^{*}} (t, x, x) \right| dx & \leq & \frac{C(\varepsilon)}{T(t)^{d-1}} \mathcal{H}^{d-1} \left( \partial D \right) \int_{0}^{\infty} \left\{ \left( 1 \wedge r^{-d - \underline{\alpha} + 1} \right) + \left( 1 \wedge r^{-d - \overline{\alpha} +1 } \right) \right\} dr \nonumber \\
 & \leq & \boxed{C(\varepsilon) T(t)^{1-d} }.
\end{eqnarray}

Finally, it remains to show that 
\begin{eqnarray}
\left| \int_{D_{2}} r_{H^{*}(x)} (t, x, x) \hspace{0.02in} dx - \mathcal{H}^{d-1}(\partial D) T(t)^{1-d}  \int_{0}^{\infty} r_{\mathbb{H}} \left( t, \left(q, 0,..., 0 \right), \left(q, 0, ..., 0 \right) \right) dq \right| \leq c(\varepsilon) T(t).
\end{eqnarray}
To do this we apply Proposition \ref{Prop1} to $\int_{D_{2}} r_{H^{*}(x)} (t, x, x) dx$. 
Note that, by construction, we have
\begin{eqnarray}
r_{H^{*}(x)} (t, x, x) & = & r_{H^{*}(x)} \left(t, \left( \delta_{H^{*}(x)} (x), 0, ..., 0 \right), \left( \delta_{H^{*}(x)} (x), 0, ..., 0 \right) \right) \\
& = & r_{H^{*}(x)} \left(t, \left( \delta_{D} (x), 0, ..., 0 \right), \left( \delta_{D} (x), 0, ..., 0 \right) \right) \\
& = & r_{\mathbb{H}} \left(t, \left( \delta_{D} (x), 0, ..., 0 \right), \left( \delta_{D} (x), 0, ..., 0 \right) \right) \\
& =: & r_{\mathbb{H}} \left(t, \delta_{D} (x) \right).
\end{eqnarray}
Thus we can change from $D_{2}$ to $D$ by remarking that 
\begin{equation}
 \int_{D_{2}} r_{H^{*}(x)} (t, x, x) \hspace{0.02in} dx = \int_{D} r_{\mathbb{H}} \left(t, \delta_{D} (x) \right) dx - \int_{D_{1} \cup D_{3}} r_{\mathbb{H}} \left(t, \delta_{D} (x) \right) dx
\end{equation}
and that by the same arguments as (\ref{EqA}) and (\ref{EqB}) we also have that
\begin{equation}
\int_{D_{1} \cup D_{3}} r_{\mathbb{H}} \left(t, \delta_{D} (x) \right) dx \leq c(\varepsilon) T(t)^{1-d},
\end{equation}
where $c(\varepsilon) \rightarrow 0$, as $\varepsilon \rightarrow 0$. Lemma~\ref{Lemma3} tells us
\begin{equation}
r_{\mathbb{H}} (t, \delta_{D} (x)) \leq C T(t)^{-d} \left( 1 \wedge \frac{T(t)^{d}}{\delta_{D}^{d}} \frac{V^{2}(T(t))}{V^{2} \left( \delta_{D} (x) \right)} \right).
\end{equation}
Applying our Potter-like bounds from Lemma~\ref{Lemma6} gives us
\begin{equation}
r_{\mathbb{H}} (t, \delta_{D} (x)) \leq \frac{C}{T(t)^{d}} \left\{ 1 \wedge \left( \frac{T(t)}{\delta_{D} (x)} \right)^{d + \underline{\alpha}} + 1 \wedge  \left( \frac{T(t)}{\delta_{D} (x)} \right)^{d + \overline{\alpha}} \right\}.
\end{equation}

We wish to show that $r_{\mathbb{H}} \left(t, \delta_{D} (x) \right)$ satisfies the assumptions of Proposition~\ref{Prop1}. Hence we must show that $r_{\mathbb{H}} \left(t, \delta_{D} (x) \right)$ is Lipschitz. Firstly, the following bound is provided by \cite{Grz}:
\begin{lemma}\label{Lemma1}
Let $\psi \in WLSC(\underline{\alpha}, \theta)$. Then for $T(t) < 1/\theta$ we have
\begin{eqnarray}
\left| \nabla_{x} p_{t} (x) \right| \leq \frac{c}{T(t)} \min \left\{ p_t(0), \frac{t}{ |x|^d V^{2}(|x|)} \right\}.
\end{eqnarray}
\end{lemma}

\noindent Next
\begin{lemma}\label{Lemma2}
Let $D \subset \mathbb{R}^d$ be an open nonempty set. Fix $\varepsilon > 0$. For any $y \in D$ and $w$, $z \in D$ with $\delta_{D}(w) > \varepsilon$, $\delta_{D} (z) > \varepsilon$, there exists $c(\varepsilon, t)$ such that
\begin{equation}
\left| r_{D} (t, w, y) - r_{D} (t, z, y) \right| \leq c(\varepsilon, t) \left| w - z \right|.
\end{equation}
\end{lemma}

\begin{proof}
The mean value theorem and Lemma~\ref{Lemma1} tells us that there exists some $0 \leq l \leq 1$ such that 
\begin{eqnarray}
\left| p_t(w) - p_t(z) \right| & \leq & \left| \nabla_x p_t ( lw + (1-l)w ) \right| |w - z| \\
& \leq & \frac{c}{T(t)} \min \left\{ p_t(0), \frac{t}{ | lw + (1 - l)z|^{d} V^{2} \left( | lw + (1 - l)z | \right)} \right\} |w - z| \\
& \leq & \frac{c}{T(t)} \min \left\{ p_t(0), \frac{t}{  ( |w| \wedge |z| )^{d} V^{2} \left( |w| \wedge |z| \right)} \right\} |w - z|.
\end{eqnarray}

\noindent By definition of the heat remainder, (\ref{Rem}), we have
\begin{equation}
r_{D} (t, x, y) = \mathbb{E}^{y} \left[ \tau_{D} < t ; \ p_{t - \tau_{D}} \left( X \left( \tau_D \right) - x \right) \right]. \nonumber
\end{equation}

\noindent Thus
\begin{eqnarray}
\left| r_{D} (t, w, y) - r_{D} (t, z, y) \right| & \leq & \mathbb{E}^{y} \left[ \tau_{D} < t ; \ p_{t - \tau_{D}} \left( X \left( \tau_D \right) - w \right) - p_{t - \tau_{D}} \left( X \left( \tau_D \right) - z \right) \right] \\
& \leq & c \mathbb{E}^{y} \bigg[ \tau_{D} < t ; \ \frac{1}{T(t - \tau_{D})} \min \bigg\{ p_{t - \tau_{D}}(0), \\
&& \hspace{0.3in} \frac{t - \tau_{D}}{  \left( |X \left( \tau_D \right) - w| \wedge |X \left( \tau_D \right) - z| \right)^{d} V^{2} \left( | X \left( \tau_D \right) - w| \wedge |X \left( \tau_D \right) - z | \right)} \bigg\} |w - z| \bigg] \nonumber \\
& \leq & c \mathbb{E}^{y} \bigg[ \tau_{D} < t ; \ \frac{1}{T(t - \tau_{D})} \min \bigg\{ p_{t - \tau_{D}}(0), \\
&& \hspace{0.3in} \frac{t - \tau_{D}}{  \left( |\delta_{D}(w)| \wedge |\delta_{D}(z)| \right)^{d} V^{2} \left( | \delta_{D}(w)| \wedge | \delta_{D}(z) | \right)} \bigg\} |w - z| \bigg] \nonumber \\
& \leq & c \frac{ | w - z | }{\left( |\delta_{D}(w)| \wedge |\delta_{D}(z)| \right)^{d} V^{2} \left( | \delta_{D}(w)| \wedge | \delta_{D}(z) | \right)} \mathbb{E}^{y} \left[ \tau_{D} < t ; \ \frac{t - \tau_{D} }{T(t - \tau_{D})} \right]  \\
& \leq & c(\varepsilon, t) | w - z |,
\end{eqnarray}
where, in the last inequality, we have used our assumption that both $\delta_{D}(w)$ and $\delta_{D} (z)$ are larger than $\varepsilon$.
\end{proof}

Finally we can now show that $r_{\mathbb{H}} \left(t, \delta_{D} (x) \right)$ is Lipschitz:
\begin{lemma} \label{Lips}
Let $D \subset \mathbb{R}^d$ be an open nonempty set. Fix $\varepsilon > 0$. For any $y \in D$ and $w$, $z \in D$ with $\delta_{D}(w) > \varepsilon$, $\delta_{D} (z) > \varepsilon$, there exists $c(\varepsilon, t)$ such that
\begin{equation}
\left| r_{D} (t, w, w) - r_{D} (t, z, z) \right| \leq c(\varepsilon, t) \left| w - z \right|.
\end{equation}
\end{lemma}
\begin{proof}
By Lemma~\ref{Lemma2} and the symmetry of the heat remainder, that is $r_{D} (t, w, z) = r_{D} (t, z, w)$, we get
\begin{eqnarray}
\left| r_{D}(t,w,w) - r_{D} (t,z,z) \right| & \leq & \left| r_{D}(t,w,w) - r_{D}(t,z,w) \right| + \left| r_{D} (t,w,z) - r_{D} (t,z,z) \right| \\
& \leq & c(\varepsilon, t) \left| w - z \right|.
\end{eqnarray}
\end{proof}

Lemma~\ref{Lips} tells us that $r_{\mathbb{H}} \left(t, \delta_{D}(x) \right)$ is Lipschitz. Thus $r_{\mathbb{H}} \left(t, \delta_{D} (x) \right)$ satisfies the assumptions of Proposition~\ref{Prop1}. Hence, for small $t$, we have
\begin{eqnarray}
\left| \int_{D} r_{\mathbb{H}} \left(t, \delta_{D} (x) \right) dx - C_{\mathbb{H}}(t) \mathcal{H}^{d-1}(\partial D) \right| \leq \boxed{\varepsilon T(t)^{1-d}}.
\end{eqnarray}

This completes the proof of Theorem~\ref{Main}.

\*

\section{Proof of Proposition~\ref{MainProp}} \label{Sec5}

\begin{proof}[{\bf Proof of Proposition~\ref{MainProp}}] We wish to show that
\begin{equation}
0 \leq r_{\mathcal{I}} (t, x, x) - r_{\mathcal{U}^{c}} (t, x, x) \leq \frac{\left(\varepsilon^{1 - \underline{\alpha}/2} + \varepsilon^{1 - \overline{\alpha}/2} \right) \vee \sqrt{\varepsilon}}{T(t)^{d}} \left( 1 \wedge \frac{T(t)^{d-1}}{\delta_{\mathcal{I}}^{d-1} (x)}  \frac{V^{2}(T(t))}{V^2\left( \delta_{\mathcal{I}} (x) \right)} \right).
\end{equation}
In order to show this inequality we combine different aspects of similar proofs given in Proposition~3.2 of \cite{Bog5} and Proposition~3.1 of \cite{Ban2}.

Firstly, by definition, we have
\begin{eqnarray}
r_{\mathcal{I}} (t, x, x) - r_{\mathcal{U}^{c}} (t, x, x) & = & p_{\mathcal{U}^{c}} (t, x, x) - p_{\mathcal{I}} (t, x, x) \\
& = & \mathbb{E}^{x} \left[ \tau_{\mathcal{I}} < t, \ X\left( \tau_{\mathcal{I}} \right) \in \mathcal{U}^{c} \hspace{0.02in} \backslash \hspace{0.02in} \mathcal{I} ; \ p_{\mathcal{U}^{c}} \left({t - \tau_{\mathcal{I}}}, X \left( \tau_{\mathcal{I}} \right), x \right) \right].
\end{eqnarray}

\noindent The space-time Ikeda-Watanabe formula from Corollary~2.8 in \cite{Kul2} then tells us that
\begin{eqnarray}
r_{\mathcal{I}} (t, x, x) - r_{\mathcal{U}^{c}} (t, x, x) & = & \int_{\mathcal{I}} \int_{0}^t p_{\mathcal{I}} (l, x, y) \int_{\mathcal{U}^{c} \backslash \mathcal{I}} \nu(y - z) \hspace{0.02in} p_{\mathcal{U}^{c}} (t-l, x, z) \hspace{0.02in} dz \hspace{0.02in} dl \hspace{0.02in} dy. \label{IWEq}
\end{eqnarray}

\noindent Without loss of generality we may assume that $q = 0$ and $v(0) = (1,0, \hdots, 0)$. Let
\begin{eqnarray}
I & = & \left\{ y \ : \ y \cdot v(0) > \varepsilon |y| \right\}, \\
U & = & \left\{ y \ : \ y \cdot v(0) < -\varepsilon |y| \right\}, \\
\Gamma(0,\varepsilon) & = & \left\{ y \ : \ y \cdot v(0) > \sqrt{1 - \varepsilon^2} |y| \right\}.
\end{eqnarray}

\begin{center}
\begin{tikzpicture}
\fill[gray!20] (-1,-1) -- (5,-1) -- (5,5) -- (-1,5);

\fill[white] (2,2) -- (3.984313,2.25) arc (7.125:172.875:2) -- cycle;
\fill[gray!40] (2,2) -- (0.0156865,1.75) arc (187.125:352.875:2) -- cycle;

\draw (2,2) node[circle, fill, inner sep=1pt] {};
\draw (2,2) node[above right] {$0$};
\draw[dotted] ({2 - sqrt(4 - (1.75 - 2)^2)},1.75) -- ({2 + sqrt(4 - (2.25 - 2)^2)},2.25);
\draw[dotted] ({2 - sqrt(4 - (1.75 - 2)^2)},2.25) -- ({2 + sqrt(4 - (2.25 - 2)^2)},1.75);

\draw (2,2) -- ({2 + sqrt(4 - (2.25 - 2)^2) - 1.3}, 2.05);
\draw ({2 + sqrt(4 - (2.25 - 2)^2) - 1.3},2.05) -- ({2 + sqrt(4 - (2.25 - 2)^2) - 1},1.92);
\draw ({2 + sqrt(4 - (2.25 - 2)^2) - 1},1.92) -- ({2 + sqrt(4 - (2.25 - 2)^2) - 0.6},2.1);
\draw ({2 + sqrt(4 - (2.25 - 2)^2) - 0.6},2.1) -- ({2 + sqrt(4 - (2.25 - 2)^2) - 0.4},1.91);
\draw ({2 + sqrt(4 - (2.25 - 2)^2) - 0.4},1.91) -- ({2 + sqrt(4 - (2.25 - 2)^2) - 0.1},2.02);
\draw ({2 + sqrt(4 - (2.25 - 2)^2) - 0.1},2.02) -- ({2 + sqrt(4 - (2.25 - 2)^2) + 0.4},1.8);
\draw ({2 + sqrt(4 - (2.25 - 2)^2) + 0.4},1.8) -- (5,1.8);

\draw (2,2) -- ({2 - sqrt(4 - (2.25 - 2)^2) + 1.3}, 1.98);
\draw ({2 - sqrt(4 - (2.25 - 2)^2) + 1.3}, 1.98) -- ({2 - sqrt(4 - (2.25 - 2)^2) + 1},2.05);
\draw ({2 - sqrt(4 - (2.25 - 2)^2) + 1},2.05) -- ({2 - sqrt(4 - (2.25 - 2)^2) + 0.6},2.1);
\draw ({2 - sqrt(4 - (2.25 - 2)^2) + 0.6},2.1) -- ({2 - sqrt(4 - (2.25 - 2)^2) + 0.4},1.91);
\draw ({2 - sqrt(4 - (2.25 - 2)^2) + 0.4},1.91) -- ({2 - sqrt(4 - (2.25 - 2)^2) + 0.1},2.07);
\draw ({2 - sqrt(4 - (2.25 - 2)^2) + 0.1},2.07) -- ({2 - sqrt(4 - (2.25 - 2)^2) - 0.4},1.9);
\draw ({2 - sqrt(4 - (2.25 - 2)^2) - 0.4},1.9) -- (-1,2.3);

\draw (0,4) node {$\mathcal{U}^{c} \hspace{0.02in} \backslash \hspace{0.02in} \mathcal{I}$};

\draw[dashed] (3.984313,2.25) arc (7.125:172.875:2);
\draw[dashed] (0.0156865,1.75) arc (187.125:352.875:2);

\draw (2,1) node {$\mathcal{I}$};

\end{tikzpicture} \hspace{0.5in}
\begin{tikzpicture}
\fill[gray!10] (2,2) -- (5,{3/8 + 2}) -- (5,5) -- (-1,5) -- (-1,{3/8 + 2});
\fill[gray!20] (2,2) -- (5,{-3/8 + 2}) -- (5,-1) -- (-1,-1) -- (-1,{-3/8 + 2});

\draw[dashed]  (2,2) circle (2cm);
\draw (2,2) node[circle, fill, inner sep=1pt] {};
\draw (2,2) node[above right] {$0$};
\draw[dotted] ({2 - sqrt(4 - (1.75 - 2)^2)},1.75) -- ({2 + sqrt(4 - (2.25 - 2)^2)},2.25);
\draw[dotted] ({2 - sqrt(4 - (1.75 - 2)^2)},2.25) -- ({2 + sqrt(4 - (2.25 - 2)^2)},1.75);

\draw (2,2) -- ({2 + sqrt(4 - (2.25 - 2)^2) - 1.3}, 2.05);
\draw ({2 + sqrt(4 - (2.25 - 2)^2) - 1.3},2.05) -- ({2 + sqrt(4 - (2.25 - 2)^2) - 1},1.92);
\draw ({2 + sqrt(4 - (2.25 - 2)^2) - 1},1.92) -- ({2 + sqrt(4 - (2.25 - 2)^2) - 0.6},2.1);
\draw ({2 + sqrt(4 - (2.25 - 2)^2) - 0.6},2.1) -- ({2 + sqrt(4 - (2.25 - 2)^2) - 0.4},1.91);
\draw ({2 + sqrt(4 - (2.25 - 2)^2) - 0.4},1.91) -- ({2 + sqrt(4 - (2.25 - 2)^2) - 0.1},2.02);
\draw ({2 + sqrt(4 - (2.25 - 2)^2) - 0.1},2.02) -- ({2 + sqrt(4 - (2.25 - 2)^2) + 0.4},1.8);
\draw ({2 + sqrt(4 - (2.25 - 2)^2) + 0.4},1.8) -- (5,1.8);

\draw (2,2) -- ({2 - sqrt(4 - (2.25 - 2)^2) + 1.3}, 1.98);
\draw ({2 - sqrt(4 - (2.25 - 2)^2) + 1.3}, 1.98) -- ({2 - sqrt(4 - (2.25 - 2)^2) + 1},2.05);
\draw ({2 - sqrt(4 - (2.25 - 2)^2) + 1},2.05) -- ({2 - sqrt(4 - (2.25 - 2)^2) + 0.6},2.1);
\draw ({2 - sqrt(4 - (2.25 - 2)^2) + 0.6},2.1) -- ({2 - sqrt(4 - (2.25 - 2)^2) + 0.4},1.91);
\draw ({2 - sqrt(4 - (2.25 - 2)^2) + 0.4},1.91) -- ({2 - sqrt(4 - (2.25 - 2)^2) + 0.1},2.07);
\draw ({2 - sqrt(4 - (2.25 - 2)^2) + 0.1},2.07) -- ({2 - sqrt(4 - (2.25 - 2)^2) - 0.4},1.9);
\draw ({2 - sqrt(4 - (2.25 - 2)^2) - 0.4},1.9) -- (-1,2.3);

\draw (0,4) node {$U$};
\draw (0,0) node {$I$};

\fill[gray!40] ({2 - 3/8}, -1) -- (2,2) -- ({2 + 3/8},-1);
\draw[densely dotted] (2,2) -- ({2 - 3/8}, -1);
\draw[densely dotted] (2,2) -- ({2 + 3/8},-1);

\draw[dashed]  (2,2) circle (2cm);

\draw ({2.37},-0.5) node[right] {$\Gamma(0,\varepsilon)$};

\draw[->] (2,2) -- node[right]{$v (0)$} (2,1);

\draw (3.5,3.2) node[above right] {$B(0,r)$};

\end{tikzpicture} 
\end{center}

\noindent Notice that
\begin{eqnarray}
&&\mathcal{U}^{c} \hspace{0.02in} \backslash \hspace{0.02in} \mathcal{I} = B^c(0,r) \cup \left(U^c \hspace{0.02in} \backslash \hspace{0.02in} I \right) \hspace{0.3in} \text{ and } \hspace{0.3in} \mathcal{I} \subset I. \nonumber 
\end{eqnarray}

\noindent Hence (\ref{IWEq}) can be broken up as
\begin{eqnarray}
r_{\mathcal{I}} (t, x, x) - r_{\mathcal{U}^{c}} (t, x, x) & \leq & \int_{\mathcal{I}} \int_{0}^t p_{\mathcal{I}} (l, x, y) \int_{(U^c \backslash I) \cap B(0,r)} \nu(y - z) \hspace{0.02in} p_{\mathcal{U}^{c}} (t - l, x, z) \hspace{0.02in} dz \hspace{0.02in} dl \hspace{0.02in} dy \nonumber \\
&& + \int_{\mathcal{I}} \int_{0}^t p_{\mathcal{I}} (l, x, y) \int_{B^c (0,r)} \nu(y - z) \hspace{0.02in} p_{\mathcal{U}^{c}} (t - l, x, z) \hspace{0.02in} dz \hspace{0.02in} dl \hspace{0.02in} dy \\ 
& = & A_{t} (x) + B_{t} (x).
\end{eqnarray}

$\bm{A_{t} (x):}$ Lemma~1.5 in \cite{Bog2} gives a bound for the heat kernel under certain scaling conditions:
\begin{lemma} \label{Heat1}
Suppose $\psi \in WLSC(\underline{\alpha},\theta)$ and $T(t) < 1/\theta$. Then there exists a constant $C$ such that 
\begin{equation}
p_t (x - z) \leq C \left(T^{-d}(t)  \wedge \frac{t}{|x - z|^d V^2 (|x - z|)} \right). \label{HeatKB}
\end{equation}
\end{lemma}

\noindent Notice that if $x \in \Gamma(0,\varepsilon)$ and $z \in U^c \hspace{0.02in} \backslash \hspace{0.02in} I = \left\{ y : -\varepsilon | y | < y \cdot v(0) < \varepsilon | y | \right\}$, then
\begin{equation}
|x - z| \geq |x| \sin \left( 2 \varphi_{\varepsilon} - \frac{\pi}{2} \right) = |x| \left( 1 - 2\cos^{2} \left( \varphi_{\varepsilon} \right) \right) = |x| (1- 2\varepsilon^{2} ). \label{Comp}
\end{equation}

\noindent Lemma~\ref{Heat1} and the monotonicity of $V(r)$ thus imply that
\begin{eqnarray}
p_{t-l}(x - z) & \leq & C  \frac{1}{|x - z|^d} \frac{t}{V^2 (|x - z|)} \nonumber \\
& \leq & C \frac{1}{\left( 1 - 2 \varepsilon^2 \right)^{d} |x|^d} \frac{t}{V^{2} \left( (1 - 2\varepsilon^2) | x| \right)}.
\end{eqnarray}

\noindent By assumption $\psi \in WUSC \left( \overline{\alpha}, \theta \right)$ and $\varepsilon < 1/4$, hence:
\begin{eqnarray}
p_{t - l}(x - z) & \leq & C \left( 1- 2\varepsilon^2 \right)^{-d -\overline{\alpha}} \frac{1}{|x|^d}  \frac{t}{V^{2} \left( | x| \right)} \leq C \frac{1}{|x|^d}  \frac{t}{V^{2} \left( | x| \right)}. 
\end{eqnarray}

\noindent We can now apply this bound directly to $A_{t}(x)$:
\begin{eqnarray}
A_{t} (x) & \leq & \int_{\mathcal{I}} \int_{0}^t p_{\mathcal{I}} (l, x, y) \int_{(U^c \backslash I) \cap B(0,r)} \nu(y - z) \hspace{0.02in} p (t - l, x, z) \hspace{0.02in} dz \hspace{0.02in} dl \hspace{0.02in} dy \\
& \leq & \frac{C}{|x|^d} \frac{t}{V^2 (|x|)} \int_{\mathcal{I}} \int_{0}^t p_{\mathcal{I}} (l, x, y) \int_{(U^c \backslash I) \cap B(0,r)} \nu(y - z) \hspace{0.02in} dz \hspace{0.02in} dl \hspace{0.02in} dy \\
& \leq & \frac{C}{|x|^d} \frac{t}{V^2 (|x|)} \int_{\mathcal{I}} \int_{0}^{V^2 (1/\theta)} p_{\mathcal{I}} (l, x, y) \hspace{0.02in} dl \hspace{0.02in} \int_{(U^c \backslash I) \cap B(0,r)} \nu(y - z) \hspace{0.02in} dz \hspace{0.02in} dy \\
& = & \frac{C}{|x|^d} \frac{t}{V^2 (|x|)} \int_{(U^c \backslash I) \cap B(0,r)} \int_{\mathcal{I}} G^{V^2(1/\theta)}_{\mathcal{I}}(x,y) \hspace{0.02in} \nu(y - z) \hspace{0.02in} dy \hspace{0.02in} dz \label{SLast} \\
& = & \frac{C}{|x|^d} \frac{t}{V^2 (|x|)} \int_{(U^c \backslash I) \cap B(0,r)} K^{V^2(1/\theta)}_{\mathcal{I}}(x,z) \hspace{0.02in} dz, \label{Last}
\end{eqnarray}

\noindent where in the last two equations we have used definitions of the truncated Green function and the truncated Poisson kernel, (\ref{Trunc}) and (\ref{Poiss}) respectively.
We can then apply the bound for truncated Poisson kernels on convex sets that is given in Lemma~2.9 of \cite{Bog5}:
\begin{eqnarray}
A_{t} (x) & \leq & \frac{C}{|x|^d} \frac{t}{V^2 (|x|)} \int_{(U^c \backslash I) \cap B(0,r)} \frac{c_\theta}{|x - z|^d} \frac{V(\delta_{\mathcal{I}} (x))}{V(\delta_{\mathcal{I}^c} (z))} \hspace{0.02in} dz.
\end{eqnarray}

\noindent Our Potter-like bounds in Lemma~\ref{Lemma6} tell us that
\begin{eqnarray}
\int\displaylimits_{(U^c \backslash I) \cap B(0,r)} \frac{1}{|x - z|^d} \frac{V(\delta_{\mathcal{I}} (x))}{V(\delta_{\mathcal{I}^c} (z))} \hspace{0.02in} dz & \leq & \int_{(U^c \backslash I) \cap B(0,r)} \frac{1}{|x - z|^d} \left\{ \left( \frac{\delta_{\mathcal{I}} (x)}{\delta_{\mathcal{I}^c} (z)} \right)^{\underline{\alpha}/2} \vee \left( \frac{\delta_{\mathcal{I}} (x)}{\delta_{\mathcal{I}^c} (z)} \right)^{\overline{\alpha}/2} \right\} dz \\
& \leq & \delta_{\mathcal{I}}^{\underline{\alpha}/2}(x) \int\displaylimits_{(U^c \backslash I) \cap B(0,r)} \frac{dz}{\delta_{\mathcal{I}^c}^{\underline{\alpha}/2} (z) |x - z|^d} + \delta_{\mathcal{I}}^{\overline{\alpha}/2}(x) \int\displaylimits_{(U^c \backslash I) \cap B(0,r)} \frac{dz}{\delta_{\mathcal{I}^c}^{\overline{\alpha}/2} (z) |x - z|^d}. \nonumber
\end{eqnarray}

\noindent In Lemma~3.2 of \cite{Ban2} it is shown that:
\begin{lemma} \label{Lemma8}
For any $\varepsilon \in (0,1/4)$, $w \in \Gamma(0,\varepsilon)$, $M \in (0,\infty]$ we have
\begin{equation}
\int_{(U^{c} \backslash I) \cap B(0,M)} \frac{dz}{\delta_{I^{c}}^{\alpha/2} (z) |z - w|^{\gamma}} \leq
\begin{cases}
c_{\gamma} \varepsilon^{1 - \alpha/2} |w|^{d - \alpha/2 - \gamma} & \text{ for } \gamma > d - \alpha/2, \\
c_{\gamma} \varepsilon^{1 - \alpha/2} M^{d - \alpha/2 - \gamma} & \text{ for } 0 < \gamma < d - \alpha/2.
\end{cases}
\end{equation}
\end{lemma}

\noindent Notice that for $z \in (U^c \backslash I) \cap B(0,r)$ we must have $\delta_{I^{c}} (z) = \delta_{\mathcal{I}^{c}} (z)$. Thus for $\gamma = d$ we get:
\begin{eqnarray}
\int\displaylimits_{(U^c \backslash I) \cap B(0,r)} \frac{1}{|x - z|^d} \frac{V(\delta_{\mathcal{I}} (x))}{V(\delta_{\mathcal{I}^c} (z))} \hspace{0.02in} dz & \leq & C \left\{ \delta_{\mathcal{I}}^{\underline{\alpha}/2}(x) \varepsilon^{1 - \underline{\alpha}/2} |x|^{-\underline{\alpha}/2} + \delta_{\mathcal{I}}^{\overline{\alpha}/2}(x) \varepsilon^{1 - \overline{\alpha}/2} |x|^{-\overline{\alpha}/2} \right\} \\
& \leq & C \left\{ \delta_{\mathcal{I}}^{\underline{\alpha}/2}(x) \varepsilon^{1 - \underline{\alpha}/2} \delta_{\mathcal{I}}^{-\underline{\alpha}/2}(x) + \delta_{\mathcal{I}}^{\overline{\alpha}/2}(x) \varepsilon^{1 - \overline{\alpha}/2} \delta_{\mathcal{I}}^{-\overline{\alpha}/2} (x) \right\} \\
& \leq & C \left\{ \varepsilon^{1 - \underline{\alpha}/2} + \varepsilon^{1 - \overline{\alpha}/2} \right\}. \label{Pois}
\end{eqnarray}

\noindent This gives us one bound for $A_{t} (x)$:
\begin{eqnarray}
A_{t} (x) & \leq & \boxed{C \left( \varepsilon^{1 - \underline{\alpha}/2} + \varepsilon^{1 - \overline{\alpha}/2} \right) \frac{1}{|x|^{d}} \frac{V^2(T(t))}{V^2 (|x|)}}.
\end{eqnarray}

\*

Let us now consider $A_{t} (x)$ from another perspective. We divide $A_{t} (x)$ into the following subregions:
\begin{eqnarray}
A_{t}(x) & = & \int_{\mathcal{I}} \int_{0}^{t/2} p_{\mathcal{I}} (l, x, y) \int_{(U^c \backslash I) \cap B(0,r)} \nu(y - z) \hspace{0.02in} p_{\mathcal{U}^{c}} (t - l, x, z) \hspace{0.02in} dz \hspace{0.02in} dl \hspace{0.02in} dy \\
&& + \int_{\mathcal{I}} \int_{t/2}^t p_{\mathcal{I}} (l, x, y) \int_{(U^c \backslash I) \cap B(0,r) \cap \{|x - z| \leq T\}} \nu(y - z) \hspace{0.02in} p_{\mathcal{U}^{c}} (t - l, x, z) \hspace{0.02in} dz \hspace{0.02in} dl \hspace{0.02in} dy \\
&& + \int_{\mathcal{I}} \int_{t/2}^t p_{\mathcal{I}} (l, x, y) \int_{(U^c \backslash I) \cap B(0,r) \cap \{|x - z| > T\}} \nu(y - z) \hspace{0.02in} p_{\mathcal{U}^{c}} (t - l, x, z) \hspace{0.02in} dz \hspace{0.02in} dl \hspace{0.02in} dy \\
&& = \mathbf{I} + \mathbf{II} + \mathbf{III}.
\end{eqnarray}

\*

\subsection{Short jump time: $\mathbf{I}$.} \* \vspace{0.1in}

For $l \in \left[ 0, t/2 \right]$ we can use the bound for the heat kernel given in (\ref{HeatKB}) of Lemma~\ref{Heat1}:
\begin{equation}
p_{\mathcal{U}^{c}} (t - l, x, z) \leq p (t - l, x, z) \leq C T(t - l)^{-d}.
\end{equation}
Monotonicity of $T(t)$ then implies
\begin{equation}
p_{\mathcal{U}^{c}} (t - l, x, z) \leq C T \left( t/2 \right)^{-d}.
\end{equation}
The scaling of $\psi(\xi)$ at infinity implies the scaling of $T(t)$ at $0$, as is shown in Lemma~2.1 of \cite{Bog5}. Hence
\begin{equation}
p_{\mathcal{U}^{c}} (t - l, x, z) \leq C \left( 1/2 \right)^{-d/\underline{\alpha}} T(t)^{-d} = CT(t)^{-d}.
\end{equation}
Thus
\begin{eqnarray}
\mathbf{I} & \leq & C T(t)^{-d} \int_{\mathcal{I}} \int_{0}^{t/2} p_{\mathcal{I}} (l, x, y) \int_{(U^c \backslash I) \cap B(0,r)} \nu(y - z) \hspace{0.02in} dz \hspace{0.02in} dl \hspace{0.02in} dy \\
& \leq & C T(t)^{-d} \int_{(U^c \backslash I) \cap B(0,r)}K_{\mathcal{I}}^{V^2(1/\theta)} (x,z) \hspace{0.02in} dz.
\end{eqnarray}

\noindent It now follows from our calculations between (\ref{Last}) and (\ref{Pois}) above that
\begin{eqnarray}
\mathbf{I} & \leq & \boxed{C \left( \varepsilon^{1 - \underline{\alpha}/2} + \varepsilon^{1 - \overline{\alpha}/2} \right) T(t)^{-d}}.
\end{eqnarray}

\*

\subsection{Long exit time and short jumps: $\mathbf{II}$.} \* \vspace{0.1in}

The following bound for the heat kernel is given in Lemma 2.6 of \cite{Bog5}:
\begin{lemma} \label{Lemma10}
Assume $D$ is convex. There exists a constant $c_{\theta}$ such that if $T(t) < 1/\theta \vee |x - y|$, then
\begin{equation}
p_{D} (t, x, y) \leq c_\theta \left( \frac{V(\delta_{D} (x))}{V(T)} \wedge 1 \right) \left( \frac{V(\delta_{D} (y))}{V(T)} \wedge 1 \right) \left( \frac{t}{|x-y|^d V^2(|x-y|)} \wedge T(t)^{-d} \right). \label{HeatBound}
\end{equation}
\end{lemma}

\noindent Let $S := (U^c \backslash I) \cap B(0,r) \cap \{|x - z| \leq T\}$. For $l \in [ t/2, t )$ we can use the bounds from Lemma~\ref{Heat1} and Lemma~\ref{Lemma10} to get
\begin{eqnarray}
\mathbf{II} & = & \int_{\mathcal{I}} \int_{t/2}^t p_{\mathcal{I}}(l, x, y) \int_{S} \nu(y - z) p_{\mathcal{U}^{c}} (t - l, x, z) \hspace{0.02in} dz \hspace{0.02in} dl \hspace{0.02in} dy \\
& \leq & C \int_{\mathcal{I}} \int_{t/2}^t T(t)^{-d} \frac{V(\delta_{\mathcal{I}} (y))}{V(T(t))} \int_{S} \frac{1}{|y - z|^d V^2 (|y - z|)} p_{\mathcal{U}^{c}} (t - l, x, z) \hspace{0.02in} dz \hspace{0.02in} dl \hspace{0.02in} dy \\
& = & C T(t)^{-d} \int_{\mathcal{I}} \int_{S} \frac{V(\delta_{\mathcal{I}}(y))}{V(T(t))} \frac{1}{|y - z|^d V^2 (|y - z|)} \int_{t/2}^t p_{\mathcal{U}^{c}} (t - l, x, z) \hspace{0.02in} dl \hspace{0.02in} dz \hspace{0.02in} dy \\
& \leq & C T(t)^{-d} \int_{S} \int_{\mathcal{I}} \frac{V(\delta_{\mathcal{I}}(y))}{V(T(t))} \frac{1}{|y - z|^d V^2 (|y - z|)} G_{\mathcal{U}^{c}}^{t/2} (x,z) \hspace{0.02in} dy \hspace{0.02in} dz.
\end{eqnarray}

\noindent It follows from bounds given in \cite{Bog2} and \cite{Bog5} that
\begin{eqnarray}
\mathbf{II} & \leq & C \frac{T(t)^{-d}}{V(T(t))} \int_{S} \int_{\mathcal{I}} \frac{V(\delta_{\mathcal{I}}(y))}{|y - z|^d V^2 (|y - z|)} \frac{V(|x|) V(\delta_{\mathcal{U}^{c}}(z))}{|x - z|^{d}} \hspace{0.02in} dy \hspace{0.02in} dz.
\end{eqnarray}

\noindent By construction $\delta_{\mathcal{I}} (y), \delta_{\mathcal{I}}(z) \leq |y - z|$ and so
\begin{eqnarray}
\mathbf{II} & \leq & C T(t)^{-d} \frac{V(|x|)}{V(T(t))} \int_{S} \int_{\mathcal{I}} \frac{1}{|y - z|^d V (|y - z|)} \frac{V(\delta_{\mathcal{U}^{c}}(z))}{|x - z|^{d}} \hspace{0.02in} dy \hspace{0.02in} dz \nonumber \\
& \leq & C T(t)^{-d} \frac{V(|x|)}{V(T(t))} \int_{S} \int_{\mathcal{I}} \frac{\delta_{\mathcal{I}}^{\underline{\alpha}/2}(z)}{|y - z|^{d+\underline{\alpha}/2} V (\delta_{\mathcal{I}}(z))} \frac{V(\delta_{\mathcal{U}^{c}}(z))}{|x - z|^{d}} \hspace{0.02in} dy \hspace{0.02in} dz \nonumber \\
& \leq & C T(t)^{-d} \frac{V(|x|)}{V(T(t))} \int_{S} \frac{\delta_{\mathcal{I}}^{\underline{\alpha}/2}(z)}{| x - z |^{d}} \frac{V(\delta_{\mathcal{U}^{c}}(z))}{V(\delta_{\mathcal{I}}(z))} \int_{\mathcal{I}} \frac{1}{|y - z|^{d+\underline{\alpha}/2}} \hspace{0.02in} dy \hspace{0.02in} dz.
\end{eqnarray}

\noindent We have seen in (\ref{Comp}) that $|x - z| > (1-2\varepsilon^2) |x|$. Thus for these short jumps we have $(1-2\varepsilon^2) |x| < T(t)$ and hence $V(|x|) < c V(T(t))$, for some constant $c$. Therefore
\begin{eqnarray}
\mathbf{II} & \leq & C T(t)^{-d} \int_{S} \frac{\delta_{\mathcal{I}}^{\underline{\alpha}/2}(z)}{| x - z|^{d}} \frac{V(\delta_{\mathcal{U}^{c}}(z))}{V(\delta_{\mathcal{I}}(z))} \int_{\mathcal{I}} \frac{1}{|y - z|^{d+\underline{\alpha}/2}} \hspace{0.02in} dy \hspace{0.02in} dz \\
& \leq & C T(t)^{-d} \int_{S} \frac{\delta_{\mathcal{I}}^{\underline{\alpha}/2}(z)}{| x - z|^{d}}  \frac{V(\delta_{\mathcal{U}^{c}}(z))}{V(\delta_{\mathcal{I}}(z))} \int_{B(z,\delta_{\mathcal{I}}(z))^c} \frac{1}{|y - z|^{d + \underline{\alpha}/2}} \hspace{0.02in} dy \hspace{0.02in} dz.
\end{eqnarray}

\noindent Changing to polar coordinates:
\begin{eqnarray}
\mathbf{II} & \leq & C T(t)^{-d} \int_{S} \frac{\delta_{\mathcal{I}}^{\underline{\alpha}/2}(z)}{| x - z|^{d}}  \frac{V(\delta_{\mathcal{U}}^{c}(z))}{V(\delta_{\mathcal{I}}(z))} \int_{\delta_{\mathcal{I}^c}(z)}^\infty \frac{1}{r^{d + \underline{\alpha}/2}} r^{d-1} \hspace{0.02in} dr \hspace{0.02in} dz \\
& = & C T(t)^{-d} \int_{S} \frac{\delta_{\mathcal{I}}^{\underline{\alpha}/2}(z)}{| x - z|^{d}}  \frac{V(\delta_{\mathcal{U}^{c}}(z))}{V(\delta_{\mathcal{I}}(z))} \frac{1}{\delta_{\mathcal{I}}^{\underline{\alpha}/2} (z)} \hspace{0.02in} dz \\
& = & C T(t)^{-d} \int_{S} \frac{1}{| x - z|^{d}}  \frac{V(\delta_{\mathcal{U}^{c}}(z))}{V(\delta_{\mathcal{I}}(z))} \hspace{0.02in} dz. \label{Num2}
\end{eqnarray}

\begin{lemma} \label{Lem5.7}
For any $\varepsilon \in (0,1/4)$, $x \in \Gamma(0,\varepsilon)$, $r \in (0, \infty)$ we have
\begin{equation}
\int_{\left( U^{c} \backslash I \right) \cap B(0,r)} \frac{1}{|x - z|^{d}} \frac{\delta_{\mathcal{U}^{c}}^{\alpha/2}(z)}{\delta_{\mathcal{I}}^{\alpha/2}(z)} \hspace{0.02in} dz \leq c \varepsilon^{1 - \alpha/2}.
\end{equation}
\end{lemma}
\begin{proof}
Let us use polar coordinates $\left( \rho, \varphi_{1}, ..., \varphi_{d} \right)$, with center $q = 0$ and principal axis $v(0) = (1, 0,..., 0)$. We prove this lemma for the case $d \geq 3$, the case with $d = 2$ is essentially the same but with different restrictions on the angle. As above, we let $\varphi_{\varepsilon} \in [ 0,\pi/2]$ be the angle such that $\cos \left( \varphi_{\varepsilon} \right) = \varepsilon$. Then $$U^{c} \backslash I = \left\{ \left( \rho, \varphi_{1}, ..., \varphi_{d-1} \right) : \varphi_{1} \in \left( \varphi_{\varepsilon}, \pi - \varphi_{\varepsilon} \right) \right\}, \hspace{0.2in}\delta_{\mathcal{I}}(z) = \rho \sin \left( \varphi_{1} - \varphi_{\varepsilon} \right), \hspace{0.2in} and \hspace{0.2in} \delta_{\mathcal{U}^{c}} (z) = \rho \sin \left( \varphi_{\varepsilon} + \varphi_{1} \right)$$ for $z \in U^{c} \backslash I$.

Let $V_{1} = \left( U^{c} \backslash I \right) \cap B(0,|x|)$ and $V_{2} = \left( U^{c} \backslash I \right) \cap B^{c}(0,|x|) \cap B(0,r)$. Recall, $(1 - 2\varepsilon^{2}) |x|, (1 - 2 \varepsilon^{2}) |z| \leq |x - z|$ and notice that for $z \in V_{1}$ we have $|x - z | \leq 2|x|$, thus $|x - z| \simeq |x|$ for $z \in V_{1}$. Similarly, if $z \in V_{2}$, then $|x - z| \simeq |z|$. Thus
\begin{eqnarray}
\int_{V_{1}} \frac{1}{|x - z|^{d}} \frac{\delta_{\mathcal{U}^{c}}^{\alpha/2}(z))}{\delta_{\mathcal{I}}^{\alpha/2}(z))} dz & \leq & \frac{c}{|x|^{d}} \int_{V_{1}}  \frac{\delta_{\mathcal{U}^{c}}^{\alpha/2}(z)}{\delta_{\mathcal{I}}^{\alpha/2}(z)} dz \\
& \leq & \frac{c}{|x|^{d}} \int_{0}^{|x|} \int_{\varphi_{\varepsilon}}^{\pi - \varphi_{\varepsilon}}  \frac{\rho^{\alpha/2} \sin \left( \varphi_{\varepsilon} + \varphi_{1} \right) \rho^{d -1} \sin^{d-2} \left( \varphi_{1} \right)}{\rho^{\alpha/2} \sin^{\alpha/2} \left( \varphi_{1} - \varphi_{\varepsilon} \right)} d\varphi_{1} d\rho \\
& \leq & \frac{c}{|x|^{d}} \int_{0}^{|x|} \rho^{d - 1} d\rho \int_{\varphi_{\varepsilon}}^{\pi - \varphi_{\varepsilon}}  \frac{1}{ \sin^{\alpha/2} \left( \varphi_{1} - \varphi_{\varepsilon} \right)} d\varphi_{1} \\
& \leq & c \int_{0}^{\pi - 2\varphi_{\varepsilon}} \frac{1}{ \varphi^{\alpha/2} } d\varphi \\
& \leq & c\varepsilon^{1 - \alpha/2}.
\end{eqnarray}

\noindent The last inequality follows from the fact that for $\varepsilon \in (0,1/4)$ we have $\sin(\pi - 2 \varphi_{\varepsilon}) \simeq 2\sin(\pi/2 - \varphi_{\varepsilon})$, so $\pi - 2\varphi_{\varepsilon} \leq c \varepsilon$. On the remaining domain we have
\begin{eqnarray}
\int_{V_{2}} \frac{1}{|x - z|^{d}} \frac{\delta_{\mathcal{U}^{c}}^{\alpha/2}(z))}{\delta_{\mathcal{I}}^{\alpha/2}(z))} dz & \leq & \int_{V_{2}} \frac{\delta_{\mathcal{U}^{c}}^{\alpha/2}(z)}{|z|^{d} \delta_{\mathcal{I}}^{\alpha/2}(z)} dz \\
& \leq & \int_{|x|}^{r} \int_{\varphi_{\varepsilon}}^{\pi - \varphi_{\varepsilon}}  \frac{\rho^{\alpha/2} \sin \left( \varphi_{\varepsilon} + \varphi_{1} \right) \rho^{d -1} \sin^{d-2} \left( \varphi_{1} \right)}{\rho^{d+\alpha/2} \sin^{\alpha/2} \left( \varphi_{1} - \varphi_{\varepsilon} \right)} d\varphi_{1} d\rho \\
& \leq & \int_{|x|}^{r} \rho^{- 1} d\rho \int_{0}^{\pi - 2\varphi_{\varepsilon}} \frac{1}{ \varphi^{\alpha/2} } d\varphi \\
& \leq & c \varepsilon^{1 - \alpha/2}.
\end{eqnarray}
\end{proof}

It now follows from (\ref{Num2}) and Lemma~\ref{Lem5.7} that
\begin{eqnarray}
\mathbf{II} & \leq & \boxed{C\left(\varepsilon^{1 - \underline{\alpha}/2} + \varepsilon^{1 - \overline{\alpha}/2} \right) T(t)^{-d}}.
\end{eqnarray}

\*

\subsection{Long exit time and large jumps: $\mathbf{III}$.} \* \vspace{0.1in}

We now suppose that $|x - z| > T$. Let $Q := (U^c \backslash I) \cap B(0,r) \cap \{|x - z| > T\}$. Again using the bound from (\ref{HeatBound}) of Lemma~\ref{Lemma10} we get
%
\begin{eqnarray}
\mathbf{III} & \leq & C \int_{\mathcal{I}} \int_{t/2}^t p_{\mathcal{I}}(l,x,y) \int_{Q} \nu (y - z) T(t-l)^{-d} \frac{V(\delta_{\mathcal{U}^{c}} (z))}{V(T(t-l))} \left( 1 \wedge \frac{T(t-l)^d V^{2}(T(t-l))}{|x - z|^{d} V^{2}( |x - z| )} \right) \hspace{0.02in} dz \hspace{0.02in} dl \hspace{0.02in} dy \nonumber \\
& \leq & C T(t)^{-d} \int_{Q} K_{\mathcal{I}}^{V^{2}(1/\theta)} (x,z) \frac{V(\delta_{\mathcal{U}^{c}} (z))}{V(T(t))} \left( 1 \wedge \frac{T(t)^d V^{2}(T(t))}{|x - z|^{d} V^{2}( |x - z| )} \right) \hspace{0.02in} dz.
\end{eqnarray}

\noindent We can again use the Poisson kernel bound from Lemma~2.9 in \cite{Bog5}:
\begin{eqnarray}
\mathbf{III} & \leq & C T(t)^{-d} \int_{Q} \frac{V(|x|)}{V(\delta_{\mathcal{I}}(z))} \frac{1}{|x - z|^{d}} \frac{V(\delta_{\mathcal{U}^{c}} (z))}{V(T(t))} \left( 1 \wedge \frac{T(t)^d V^{2}(T(t))}{|x - z|^{d} V^{2}( |x - z| )} \right) \hspace{0.02in} dz \\
& \leq & C V(T(t)) \int_{Q} \frac{V(|x|)}{|x - z|^{2d} V^{2} (|x - z|)} \frac{V(\delta_{\mathcal{U}^{c}} (z))}{V(\delta_{\mathcal{I}} (z))} \hspace{0.02in} dz \\
& \leq & C V(T(t)) \int_{Q} \frac{1}{|x - z|^{2d} V (|x - z|)} \frac{V(\delta_{\mathcal{U}^{c}} (z))}{V(\delta_{\mathcal{I}} (z))} \hspace{0.02in} dz \\
& \leq & C \frac{V(T(t))}{\left( T(t) \right)^{d} V \left( T(t) \right)} \int_{Q} \frac{V(\delta_{\mathcal{U}^{c}} (z))}{|x - z|^{d} V(\delta_{\mathcal{I}} (z))} \hspace{0.02in} dz.
\end{eqnarray}

\nonumber We can use Lemma~\ref{Lem5.7} again to get
\begin{equation}
\mathbf{III} \leq \boxed{C\left(\varepsilon^{1 - \underline{\alpha}/2} + \varepsilon^{1 - \overline{\alpha}/2} \right) T(t)^{-d}}.
\end{equation}

Therefore
\begin{eqnarray}
A_{t} (x) & \leq & \boxed{C\left(\varepsilon^{1 - \underline{\alpha}/2} + \varepsilon^{1 - \overline{\alpha}/2} \right) \left( T(t)^{-d} \wedge  \frac{1}{|x|^d}\frac{V^{2}(T(t))}{V^{2}(|x|)} \right)}.
\end{eqnarray}

\*

$\bm{B_{t}(x)}$:
It remains to find a bound for
\begin{eqnarray}
B_{t} (x) \leq \int_{\mathcal{I}} \int_{0}^t p_{\mathcal{I}} (l, x, y) \int_{B^c (0,r)} \nu(y - z) p (t - l, x, z) \hspace{0.02in} dz \hspace{0.02in} dl \hspace{0.02in} dy.
\end{eqnarray}

\noindent By assumption $x \in \Gamma_{2s}(v(q),\varepsilon)$, $s < r/4$, and $z \in B^c (0,r)$. Thus $|x - z| > r/2 > 2s$. Combining this with the bound for the heat kernel in Lemma~\ref{Heat1}, we get:
\begin{eqnarray}
p (t - l, x, z) & \leq & C \left( T(t - l)^{-d} \wedge \frac{1}{|x - z|^d} \frac{t - l}{V^2 (|x - z|)} \right) \\
& \leq & C \left( T(t - l)^{-d} \wedge \frac{1}{s^d} \frac{t - l}{V^2(s)} \right). 
\end{eqnarray}

Thus
\begin{eqnarray}
B_{t} (x) & \leq & C \left( T(t - l)^{-d} \wedge \frac{1}{s^d} \frac{V^{2}(T(t))}{V^2(s)} \right) \int_{\mathcal{I}} \int_{0}^t p_{\mathcal{I}}(l,x,y) \int_{B^c (0,r)} \nu(y - z) dz dl dy  \\ 
& \leq & C \left( T(t - l)^{-d} \wedge \frac{1}{s^d} \frac{V^{2}(T(t))}{V^2(s)} \right) \mathbb{P}^{x} \left( \tau_{\mathcal{I}} < t, \left| X \left( \tau_{\mathcal{I}} \right) \right| > r \right) \\
& \leq & C \frac{1}{s^d} \frac{V^{2}(T(t))}{V^2(s)}. 
\end{eqnarray}

We chose $s = T(t)/\sqrt{\varepsilon}$. Thus
\begin{eqnarray}
B_{t} (x) & \leq & C \frac{\left( \sqrt{\varepsilon} \right)^{d} }{T(t)^d} \frac{V^{2}(T(t))}{V^2 \left( \frac{T(t)}{\sqrt{\varepsilon} }\right)}.
\end{eqnarray}

Since $x \in \Gamma_{2s}(v(q),\varepsilon)$ it also tells us that $|x| < 2s = 2T(t)/\sqrt{\varepsilon}$. Hence
\begin{equation}
B_{t} (x) \leq C \frac{\left( \sqrt{\varepsilon} \right)^{\beta}}{T(t)^{\beta}} \frac{1}{\left| x \right|^{d - \beta}} \frac{V^{2}(T(t))}{V^2 \left( \frac{T(t)}{\sqrt{\varepsilon} }\right)} \leq C \frac{\left( \sqrt{\varepsilon} \right)^{\beta}}{T(t)^{\beta}} \frac{1}{\left| x \right|^{d - \beta}} \frac{V^{2}(T(t))}{V^2 \left( |x| \right)}.
\end{equation}

Letting $\beta = d$ and $\beta = 1$ we get
\begin{equation}
B_{t} (x) \leq C \left( \frac{\left( \sqrt{\varepsilon} \right)^{d}}{T(t)^{d}} \wedge \frac{\sqrt{\varepsilon} }{T(t)} \frac{1}{\left| x \right|^{d - 1}} \frac{V^{2}(T(t))}{V^2 \left( |x| \right) } \right) \leq \boxed{C \sqrt{\varepsilon} T(t)^{-d} \left( 1 \wedge \frac{T(t)^{d-1}}{|x|^{d-1}} \frac{V^{2} \left(T(t) \right)}{V^{2} \left( |x| \right)} \right)}.
\end{equation}

Therefore, combining our bounds for $A_{t}(x)$ and $B_{t}(x)$, we get
\begin{eqnarray}
r_{\mathcal{I}} (t,x,x) - r_{\mathcal{U}^{c}} (t,x,x) \leq \boxed{C \left\{ \left(\varepsilon^{1 - \underline{\alpha}/2} + \varepsilon^{1 - \overline{\alpha}/2} \right) \vee \sqrt{\varepsilon} \right\} T(t)^{-d} \left( 1 \wedge \frac{T(t)^{d-1}}{|x|^{d-1}} \frac{V^{2}(T(t))}{V^2(|x|)} \right)}.
\end{eqnarray}

\end{proof}

\vfill

\end{document}